\documentclass{article}

\usepackage{microtype}
\usepackage{graphicx}
\usepackage{subfigure}
\usepackage{booktabs} 
\usepackage{bm}
\usepackage{hyperref} 
\usepackage{comment}
\usepackage{multirow}
\usepackage[page]{appendix}
\usepackage[flushleft]{threeparttable}
\usepackage{makecell}
\usepackage{enumerate}
\usepackage{diagbox}
\usepackage{bbding}
\usepackage{bbm}

\usepackage{bookmark}

\usepackage[accepted]{icml2023}

\usepackage{amsmath}
\usepackage{amssymb}
\usepackage{mathtools}
\usepackage{amsthm}

\theoremstyle{plain}
\newtheorem{theorem}{Theorem}[section] 
\newtheorem{lemma}[theorem]{Lemma}

\newtheorem{proposition}[theorem]{Proposition}
\theoremstyle{definition}
\newtheorem{definition}[theorem]{Definition}

\theoremstyle{remark}
\newtheorem{remark}[theorem]{Remark}

\icmltitlerunning{Projected Tensor Power Method for Hypergraph Community Recovery}

\newcommand{\Br}{\mathbb{R}}

\newcommand{\BE}{\mathbb{E}}

\newcommand{\BM}{\mathbb{M}}

\newcommand{\BH}{\bm{H}}
\newcommand{\BP}{\bm{P}}
\newcommand{\BQ}{\bm{Q}}
\newcommand{\BA}{\bm{A}}

\newcommand{\BW}{\bm{W}}

\newcommand{\LCal}{\mathcal{L}}
\newcommand{\BCal}{\mathcal{B}}
\newcommand{\CCal}{\mathcal{C}}

\newcommand{\GCal}{\mathcal{G}}
\newcommand{\HCal}{\mathcal{H}}

\newcommand{\OCal}{\mathcal{O}}

\newcommand{\MCal}{\mathcal{M}}

\newcommand{\TCal}{\mathcal{T}}
\newcommand{\XCal}{\mathcal{X}}
\newcommand{\YCal}{\mathcal{Y}}
\newcommand{\ICal}{\mathcal{I}}

\newcommand{\ba}{\begin{array}}
	\newcommand{\ea}{\end{array}}
\newcommand{\ACal}{\mathcal{A}}

\newcommand{\RCal}{\mathcal{R}}

\newcommand{\EE}{{\mathbb{E}}}

\def\ba{\bm{a}}

\def\bx{\bm{x}}
\def\by{\bm{y}}
\def\bh{\bm{h}}

\def\b1{\bm{1}}

\begin{document}

\twocolumn[
\icmltitle{Projected Tensor Power Method for Hypergraph Community Recovery}




\begin{icmlauthorlist}
	\icmlauthor{Jinxin Wang}{to}	
	\icmlauthor{Yuen-Man Pun}{goo}
	\icmlauthor{Xiaolu Wang}{to}
	\icmlauthor{Peng Wang}{ed}
	\icmlauthor{Anthony Man-Cho So}{to}
\end{icmlauthorlist}
\icmlaffiliation{to}{Department of Systems Engineering and Engineering Management, The Chinese University of Hong Kong, Hong Kong}
\icmlaffiliation{goo}{CIICADA Lab, School of Engineering, Australian National University, Canberra}
\icmlaffiliation{ed}{Department of Electrical Engineering and Computer Science, University of Michigan, Ann Arbor}
\icmlcorrespondingauthor{Anthony Man-Cho So}{manchoso@se.cuhk.edu.hk}
 
\icmlkeywords{Machine Learning, ICML}

\vskip 0.3in
]



\printAffiliationsAndNotice{}  

\begin{abstract}
    This paper investigates the problem of exact community recovery in the symmetric $d$-uniform $(d \geq 2)$ hypergraph stochastic block model ($d$-HSBM). In this model, a $d$-uniform hypergraph with $n$ nodes is generated by first partitioning the $n$ nodes into $K\geq 2$ equal-sized disjoint communities and then generating hyperedges with a probability that depends on the community memberships of $d$ nodes. Despite the non-convex and discrete nature of the maximum likelihood estimation problem, we develop a simple yet efficient iterative method, called the \emph{projected tensor power method}, to tackle it. As long as the initialization satisfies a partial recovery condition in the logarithmic degree regime of the problem, we show that our proposed method can exactly recover the hidden community structure down to the information-theoretic limit with high probability. Moreover, our proposed method exhibits a competitive time complexity of $\OCal(n\log^2n/\log\log n)$ when the aforementioned initialization condition is met. We also conduct numerical experiments to validate our theoretical findings.
\end{abstract}

	\section{Introduction}
Community detection (also known as graph clustering) is a fundamental task in various scientific and engineering fields ranging from data mining \citep{cabreros2016detecting,shi2000normalized} to network analysis \citep{girvan2002community}. One celebrated and perhaps the simplest probabilistic model for generating random graphs with community structure is the stochastic block model (SBM) \citep{holland1983stochastic}, which tends to exhibit more edge connections in communities and fewer edge connections across communities. Over the past decades, SBM has served as an important benchmark for validating and comparing various statistical theories and computational methods for community detection under different settings; see \citet{abbe2017community} for a recent comprehensive survey of SBM. Despite the great success of SBM achieved on graph data, pairwise interactions represented by SBM are inadequate for modeling complex relational information in many real-world applications. For example, in social/academic networks, many cooperative relations like chat groups and co-author lists may consist of more than two people. Other applications involving such kind of high-order relations include congress voting
networks \citep{lee2017time}, molecular interaction
networks \citep{michoel2012alignment}, as well as high-order graph matching \citep{duchenne2011tensor}. Hence, it is natural and of keen interest to study an analogous model of SBM for capturing the aforementioned high-order relations.

In this work, to capture high-order interactions among multiple objects, we focus on the symmetric $d$-uniform hypergraph stochastic block model ($d$-HSBM) \citep{ghoshdastidar2014consistency, ghoshdastidar2017consistency, kim2018stochastic, chien2018community, ahn2018hypergraph, ke2019community, cole2020exact, zhang2022exact}---a natural extension of the SBM---and study the problem of exact community recovery. Specifically, in the $d$-HSBM, $n$ nodes are partitioned into $K \geq 2$ unknown equal-sized non-overlapping communities, and each subset of nodes with cardinality $d$ independently forms an order-$d$ hyperedge with probability $p$ if these $d$ nodes are in the same community and with probability $q$ otherwise. The goal is to identify the underlying community structure exactly based on a realization of such a random hypergraph. In the logarithmic degree regime of the $d$-HSBM, i.e., $p=\alpha \log n/n^{d-1}$ and $q =\beta \log n/{n^{d-1}}$ for some $\alpha>\beta >0$, it has been recently established in \citet{kim2017community,kim2018stochastic,zhang2022exact} that there exists a sharp phase transition around a threshold: It is possible to exactly identify the underlying communities with high probability if $\frac{(\sqrt{\alpha} - \sqrt{\beta})^2}{K^{d-1}(d-1)!} >1$ and is impossible to recover the communities with non-vanishing probability if $\frac{(\sqrt{\alpha} - \sqrt{\beta})^2}{K^{d-1}(d-1)!} <1$. On top of this breakthrough, a natural question arises: Can we design a computationally tractable algorithm that achieves exact recovery down to the aforementioned information-theoretic limit? In the past few decades, many computational methods have been developed in addressing this question, such as spectral clustering methods \citep{ghoshdastidar2015provable,ghoshdastidar2017consistency, chien2018community, ahn2018hypergraph,zhang2022exact}, semidefinite programming-based methods \citep{kim2018stochastic,gaudio2022community}, and tensor decomposition-based methods \citep{ke2019community,han2022exact}. However, some of these methods lack theoretical guarantees for exact recovery with high probability at the information-theoretic limit. Moreover, most of these methods have a time complexity of at least $\OCal(n^2)$, which is less favorable in contemporary large-scale problems.

In the symmetric $d$-HSBM with $K \geq 2$ communities, the maximum likelihood (ML) estimation problem takes the form
\begin{align} \label{eq: MLE}
	\max_{\BH \in \mathbb{R}^{n \times K}} \;  \left\{ \left \langle \ACal,  \bm{H}^{\otimes d} \right\rangle : \;  \bm{H} \in \HCal \right\}. \tag{MLE}
\end{align}
Here, $\ACal$ is the adjacency tensor of the observed hypergraph (see its definition in \eqref{eq: defdhsbm}), $\BH^{\otimes d}$ is the outer product of the matrix $\bm{H}$ (see its definition in \eqref{eq: defotimesH}),
\begin{align*}
	\HCal=\left\{ \bm{H} \in \{0,1\}^{n \times K}: \bm{H} \b1_{K}=\b1_{n}, \bm{H}^\top \b1_{n}=m \b1_{K}\right\}
\end{align*}
is the discrete feasible set characterizing the possible community assignment of $n$ nodes into $K$ clusters, $\b1_n$ (resp. $\b1_K$) is the all-one vector of dimension $n$ (resp. $K$), and $m=n/K$ is the number of nodes in each community.
 It is known that an ML estimator can achieve exact recovery with high probability down to the information-theoretic limit; see, e.g., \citet[Proposition 1; Theorem 1]{kim2018stochastic}. Although solving the non-convex problem \eqref{eq: MLE} is NP-hard in the worst case, recent advances in different applications of non-convex optimization, including phase retrieval \citep{candes2015phase,chen2019gradient}, low-rank matrix recovery \citep{chi2019nonconvex,li2020nonconvex}, low-rank tensor decomposition \citep{richard2014statistical,huang2022power,han2022exact}, phase/group synchronization \citep{liu2017estimation,zhong2018near,zhu2021orthogonal,ling2022improved}, community detection \citep{wang2021optimal,wang2021non,wang2020nearly,wang2022exact}, and graph matching \citep{araya2022seeded}, suggest that it could be possible to develop some simple iterative method that solves problem \eqref{eq: MLE} down to the information-theoretic limit. In this work, we propose a simple and scalable method, called \emph{projected tensor power method} (PTPM), to tackle the discrete optimization problem \eqref{eq: MLE} and establish its exact recovery guarantee down to the information-theoretic limit. In contrast to the existing works on \emph{generalized power methods} for solving non-convex optimization problems where theoretical analyses are performed upon a \emph{quadratic} objective \citep{liu2017estimation,zhong2018near,zhu2021orthogonal,ling2022improved,wang2021optimal,wang2021non,wang2022exact,araya2022seeded}, it is worth highlighting that this paper establishes global optimality and fast convergence rate of PTPM for a \emph{polynomial} optimization problem. Thus, our work expands the repertoire of globally solvable non-convex optimization problems by generalized power methods.

\subsection{Related Literature}

There are several different goals for community detection over hypergraphs. One is \emph{exact recovery} (also known as \emph{strong consistency}), which is to identify the true underlying community structures with high probability based on a realization of a random hypergraph. Regarding the associated information-theoretic limit, in the logarithmic degree regime of the symmetric $d$-HSBM with $K=2$, it was proved in \citet[Theorem 1]{kim2018stochastic} that exact recovery is achievable if and only if $\frac{(\sqrt{\alpha} - \sqrt{\beta})^2}{2^{d-1}(d-1)!} >1$. Later, the exact recovery threshold $\frac{(\sqrt{\alpha} - \sqrt{\beta})^2}{K^{d-1}(d-1)!} >1$ was extended to scenarios with $K \geq 2$ in \citet[Theorem 2]{zhang2022exact}. Another two goals are \emph{almost exact recovery} (also known as \emph{weak consistency}) and \emph{partial recovery}, respectively. The former aims at identifying the true communities with a vanishing fraction of misclassified vertices, while the latter merely aims at correctly identifying a constant fraction of vertices; see \citet{abbe2017community} for further details of these goals.

Apart from the information-theoretic limits mentioned above, many efforts have been made to develop algorithms for the exact recovery of the $d$-HSBM over the past few years.
\vspace{-0.3cm}
\paragraph{Spectral methods.}
One popular method is spectral clustering, which generally involves three steps: (i) constructing a data matrix, (ii) performing eigendecomposition of the data matrix, and (iii) applying the $k$-means clustering algorithm to the eigenvectors. For example, in \citet{ghoshdastidar2014consistency,ghoshdastidar2015provable,ghoshdastidar2017consistency}, the authors first constructed a weight matrix based on either the hypergraph Laplacian or the tensor unfolding, and then applied the $k$-means clustering algorithms to the leading $K$ eigenvectors of the obtained weight matrix. However, the theoretical results therein require the hypergraph to be dense, and hence, the condition for exact recovery is not optimal. Moreover, spectral clustering methods generally require polynomial running time.
Recently, \citet{gaudio2022community} proved that the spectral method based on the weight matrix (also called the similarity matrix) of a constructed weighted graph can already achieve exact recovery without performing $k$-means clustering. Although the approach of projecting the original hypergraph into a weighted graph allows one to directly apply existing methods for graph networks, such a method is not optimal for exact recovery due to the loss of information \citep{ke2019community} when constructing the similarity matrix from the observed adjacency tensor. 
\vspace{-0.3cm}
\paragraph{Semidefinite programming (SDP)-based methods.} 
The SDP-based methods with $K=2$ have been considered in \citet{kim2018stochastic,gaudio2022community}. However, the conditions for achieving exact recovery are not optimal as the SDP-based methods only utilize the similarity matrix instead of the original adjacency tensor. Moreover, solving large-scale SDP is usually computationally heavy. 
\vspace{-0.3cm}
\paragraph{Two-stage methods.} It was suggested in \citet{abbe2017community,kim2018stochastic} that a local refinement method together with an initialization satisfying partial recovery can possibly lead to exact recovery. Building on such a high-level idea, there are some follow-up works for tackling the problem via two-stage methods. \citet{chien2019minimax} considered a two-stage algorithm that starts from a weakly consistent initialization and then refines it by a local maximum likelihood estimation method for each node separately. It recovers the communities exactly with high probability down
to information-theoretic threshold in $\OCal(n^3 \log n)$ time; see \citet[Section \uppercase\expandafter{\romannumeral4}.C]{chien2019minimax}. \citet{zhang2022exact} also proposed a two-stage algorithm, in which a hypergraph spectral clustering step in the first stage ensures weak consistency and a follow-up local refinement stage guarantees exact recovery. Their proposed method achieves exact recovery at the information-theoretic limit in polynomial time. In addition, \citet{ke2019community} considered a degree-corrected hypergraph SBM and developed a two-step approach including a regularized high-order orthogonal iteration algorithm and the $k$-means clustering starting from a suitable initialization. Their proposed method generalizes the celebrated tensor power method for tensor principal component analysis (PCA) \citep{richard2014statistical,huang2022power}. The setting therein is more general as it can deal with degree heterogeneity. Nevertheless, their theoretical results are not optimal when applied to the symmetric $d$-HSBM \citep[Corollary 1]{ke2019community}. Other interesting two-stage methods include a high-order extension of the Lloyd algorithm for clustering under the general tensor block model \citep{han2022exact}.

We summarize the above related works in Table \ref{table-0}.
\vskip -0.1in
\begin{table}[!htp]
\setlength\tabcolsep{4pt}
	\caption{Comparison of recovery conditions and time complexities of the discussed methods for exact recovery in the $d$-HSBM ($K\ge 2$).}
	\label{table-0}
	\vskip 0.15in
	{\small 	\begin{center}
			\begin{tabular}{ccccc}
				\toprule
				{\bf References} & {\bf Optimal} &  {\bf Complexities} \\
				\midrule
				\citet{ghoshdastidar2015provable} & \XSolidBrush & Polynomial \\ 
				\makecell[c]{\citet{gaudio2022community}\\ (Spectral method)} & \XSolidBrush & Polynomial \\ 
				\midrule \makecell[c]{\citet{kim2018stochastic}, \\ 
					\citet{gaudio2022community} \\ (SDP)} & \XSolidBrush & Polynomial\\
				\midrule
				\citet{chien2019minimax} & \CheckmarkBold & $\OCal(n^3 \log n)$ \\
				\citet{zhang2022exact} & \CheckmarkBold & Polynomial \\
				\midrule
				\textbf{Ours} & \CheckmarkBold&  $\OCal\left(\frac{n\log^2n}{\log\log n}\right)$\\
				\bottomrule
			\end{tabular}
	\end{center} }
\end{table}

\subsection{Our Contributions}
In this work, we tackle the non-convex discrete optimization problem \eqref{eq: MLE} via a simple and scalable projected tensor power method. Specifically, given an initialization satisfying a certain partial recovery condition, we refine the estimate via projected tensor power iteration successively. In the logarithmic degree regime of the $d$-HSBM, we prove that PTPM can exactly recover the underlying community labels within $\OCal(\log n/ \log \log n)$ iterations with high probability at the information-theoretic limit. Moreover, each iteration requires only $\OCal(n \log n)$ time. Therefore, the overall time complexity of PTPM given a qualified initialization would be $\OCal(n\log^2n/\log\log n)$, which is competitive to the state-of-the-art methods. Besides the simplicity of PTPM, we remark that it only requires an initialization satisfying the partial recovery condition, which is much milder than the conditions imposed in the majority of existing two-stage methods; see, e.g., \citet{chien2019minimax} and \citep{zhang2022exact}. As a result, we provide an affirmative answer to the question raised in \citet{abbe2017community,kim2018stochastic} that whether a local refinement method together with an initialization only satisfying the partial recovery can lead to exact recovery. 
As a byproduct of our analysis, we leverage the Kahn-Szemer{\'e}di argument and provide a novel concentration inequality for dealing with the extremely sparse adjacency tensor (see Lemma \ref{prop: concentration}). Our bound is much tighter than existing concentration bounds for tensors (see, e.g., \citet[Theorem 2.3; Remark 2.2; Lemma 6.1]{zhou2021sparse}) and thus can be of independent interest to applications with high-order relations.

Our work also contributes to the emerging provable non-convex optimization area. In particular, despite the non-convex and discrete nature of problem \eqref{eq: MLE}, our proposed PTPM solves the problem efficiently and optimally given a carefully designed initialization. Moreover, prior to our work, the analyses of the generalized power method are performed on non-convex optimization problems with quadratic objectives; see, e.g., \citet{boumal2016nonconvex,chen2018projected,zhong2018near,zhu2021orthogonal,wang2021non,araya2022seeded}. By sharp contrast, our analysis is performed on an optimization problem with a \emph{polynomial} objective, which expands the repertoire of globally solvable non-convex optimization problems.

The rest of this paper is organized as follows. In Section \ref{sec: notation-prelim}, we review some basic concepts in tensor algebra that will be used throughout the paper. Next, we introduce the $d$-HSBM and the proposed PTPM in Section \ref{sec: PTPM} and present our main results in Section \ref{sec: mainresults}. We then report some numerical results in Section \ref{sec: experiments} and conclude in Section \ref{sec: conclusion}.

\section{Notation and Preliminaries} \label{sec: notation-prelim}
\label{sec: nota-prelim}
We use bold uppercase letters $\bm{A},\bm{B},\dots$ to denote matrices and $\bm{A}_j$ and $A_{ij}$ to denote the $j$-th column and the $(i,j)$-th entry of $\bm{A}$, respectively. We use calligraphic letters $\ACal,\BCal,\dots$ to denote tensors of order three or higher. For instance, an order-$d$ tensor $\ACal \in \Br^{p_1 \times \dots \times p_d}$ represents a $d$-way array of size $p_1\times p_2 \times \dots \times p_d$. The $(i_1,\dots,i_d)$-th entry of a tensor $\ACal$ is denoted by $\ACal_{i_1,\dots,i_d}$. If $p_1=p_2=\dots=p_d =n$, we simply write $\ACal \in T^d(\Br^n)$. We say a tensor $\ACal \in T^d(\Br^n)$ is \emph{symmetric} if $\ACal_{i_1,\dots,i_d} = \ACal_{j_1,\dots,j_d}$ whenever $(j_1,\dots,j_d)$ is a permutation of $(i_1,\dots,i_d)$ and in this case we write $\ACal \in S^d(\Br^n)$. The \emph{inner product} of two tensors $\XCal,\YCal$ with the same dimension is defined as $\langle \XCal,\YCal \rangle = \sum_{i_1,\dots,i_d} \XCal_{i_1,\dots,i_d} \YCal_{i_1,\dots,i_d}$. The \emph{Frobenius norm} of a tensor $\XCal$ is defined as $\| \XCal \|_F = \langle \XCal,\XCal \rangle^{1/2}$. The \emph{multilinear multiplication} of a tensor $\XCal \in \Br^{r_1\times \dots \times r_d}$ by matrices $\bm{U}_k \in \Br^{p_k \times r_k}$ for $k=1,\dots,d$ is defined as
\begin{align*}
	& \quad \left(\XCal \times_{1} \mathbf{U}_{1} \times_2 \cdots  \times_{d} \mathbf{U}_{d}\right)_{i_{1}, \ldots, i_{d}}\\
	& =\sum_{j_{1}=1}^{r_{1}} \cdots \sum_{j_{d}=1}^{r_{d}} \mathcal{X}_{j_{1}, \ldots, j_{d}}\left(\mathbf{U}_{1}\right)_{i_{1} j_{1}} \cdots\left(\mathbf{U}_{d}\right)_{i_{d} j_{d}},
\end{align*}
which outputs an order-$d$ $(p_1,\dots,p_d)$-dimensional tensor. For a vector $\bx \in \Br^n$, the outer product $\bx^{\otimes d}$ is a tensor $\XCal \in S^d(\Br^n)$ with
\begin{align} \label{eq: defotimes}
	\XCal_{i_1,\dots,i_d} = x_{i_1}\times \dots \times  x_{i_d}.
\end{align}
For a matrix $\bm{H} \in \Br^{n \times K}$, the outer product $\bm{H}^{\otimes d}$ is a tensor $\XCal \in S^d(\Br^n)$ \citep[Eq. (3.4)]{kolda2009tensor} with 
\begin{equation}\label{eq: defotimesH}
	\XCal_{i_1,\dots,i_d} = \sum_{k=1}^K H_{i_1 k}\times \cdots \times H_{i_d k}.
\end{equation}
Equivalently, $\XCal = \sum_{k=1}^K (\BH_k)^{\otimes d}$.
For a tensor $\ACal \in S^d(\Br^n)$, the \emph{multilinear operation} $\ACal\left[\bm{H}^{\otimes (d-1)}\right]$ (see \citet[Eq. (2); Section 1.2]{huang2022power} and \citet{richard2014statistical}) outputs an $ n \times K$ matrix with its $(i,k)$-th entry given by
\begin{align} \label{eq: grad-tensor-def}
	& \quad \, \left( \ACal\left[{\bm{H}}^{\otimes (d-1)}\right] \right)_{ik} \nonumber\\
	& = \sum_{1 \leq i_2,\dots,i_d \leq n} \left(\ACal_{i,i_2,\dots,i_d} \times H_{i_2k} \times \cdots \times H_{i_dk} \right) \nonumber \\
	& = \left \langle \ACal, \bm{e}_i \otimes (\bm{H}_k)^{\otimes (d-1)} \right\rangle,
\end{align}
where $\bm{e}_i \in \Br^n$ stands for the zero vector except for the $i$-th entry being one and the outer product operation $\bm{e}_i \otimes (\bm{H}_k)^{\otimes (d-1)}$ is defined as
\begin{align}
	&\left(\bm{e}_i \otimes (\bm{H}_k)^{\otimes (d-1)} \right)_{j,i_2,\dots,i_d} \nonumber \\
    = \;&(\bm{e}_i)_j \cdot \left((\bm{H}_k)^{\otimes (d-1)} \right)_{i_2,\dots,i_d}
\end{align}
for $j=1,\dots,n$ and $1 \leq i_2,\dots,i_d \leq n$.
For a tensor $\ACal \in S^d(\mathbb{R}^n)$, the \emph{mode-1 matricization} of $\ACal$, denoted by $\mathcal{M}(\ACal) \in \Br^{n \times n^{d-1}}$, is defined as
\begin{equation}\label{eq: matricization}
	( \MCal(\ACal))_{ij} = \ACal_{i,i_2,\dots,i_d} \text{ with } j = 1 + \sum_{k=2}^d (i_k-1)n^{k-2};
\end{equation}
see, e.g., \citet[Section 2.4]{kolda2009tensor} and \citet[Section 2.1]{han2022exact}.
For a matrix $\bm{H} \in \Br^{n \times K}$, we use $\bm{H}^{\odot (d-1)} \in \Br^{n^{d-1} \times K}$ to represent the \emph{Khatri-Rao product} (also known as the column-wise Kronecker product) of $\BH$, which is defined as
\begin{equation}\label{eq: krproduct}
	\left( \bm{H}^{\odot (d-1)} \right)_{jk} = H_{i_2k} \times H_{i_{3}k} \times \dots \times H_{i_dk} 
\end{equation}
with $j = 1 + \sum_{k=2}^d (i_k-1)n^{k-2}$.
Combining \eqref{eq: matricization} and \eqref{eq: krproduct} yields a useful fact 
\begin{align} \label{eq: grad-two-eq-expression}
	\ACal\left[{\bm{H}}^{\otimes (d-1)}\right] = \MCal(\ACal) \left( \bm{H}^{\odot (d-1)} \right).
\end{align}
The readers are referred to \citet{kolda2009tensor,sidiropoulos2017tensor,cichocki2016tensor} for a more detailed introduction to tensor algebra. In addition, we use $\Pi_K$ to denote the collections of all $K \times K$ permutation matrices and use $\mathbf{Bern}(p)$ to denote the Bernoulli random variable with parameter $p$. Given a positive integer $n$, we denote by $[n]$ the set $\{1,\dots,n\}$. Given a discrete set $S$, we denote by $|S|$ the cardinality of $S$. If two random variables $X$ and $Y$ are equal in distribution, we write $X \overset{\textnormal{d}}{=}
Y$.  

\section{Projected Tensor Power Method} \label{sec: PTPM}
We formally state the symmetric $d$-HSBM in the following definition.
\begin{definition}[Symmetric $d$-HSBM] \label{def: dHSBM}
	Let $n \ge 2$ be the number of nodes, $K \ge 2$ be the number of communities, and $p,q\in (0,1]$ be the probability parameters of generating hyperedges. Furthermore, let $\BH^* \in \HCal$ represent a hidden partition of $n$ nodes into $K$ equal-sized disjoint communities. A random hypergraph is generated according to the symmetric $d$-HSBM with parameters $(n,d,K,p,q)$ and $\BH^*$ if the adjacency tensor $\ACal \in T^d(\Br^n)$ of such a hypergraph is symmetric and the elements 
 \begin{equation*}
     \{ 	\ACal_{i_1,i_2,\dots,i_d}\}_{1\leq i_1 <i_2<\dots<i_d \leq n}
 \end{equation*}
 are generated independently by
{\small 	\begin{equation}\label{eq: defdhsbm}
			\ACal_{i_1,i_2,\dots,i_d} \sim \begin{cases}\mathbf{Bern}(p),&\text{if } \left(\bh^*_{i_1}\circ \bh^*_{i_2}\circ \dots \circ \bh^*_{i_d}\right)\bm{1}_K =1, \\ \mathbf{Bern}(q),&\text{if }  \left(\bh^*_{i_1}\circ \bh^*_{i_2}\circ \dots \circ \bh^*_{i_d}\right) \bm{1}_K =0,\end{cases}
	\end{equation}}where $\bh_i^*$ is the $i$-th row of $\bm{H}^*$ and ``$\circ$" denotes the Hadamard (i.e., element-wise) product. In addition, since each $d$-uniform hyperedge consists of exactly $d$ nodes, the diagonal elements of $\ACal$ are automatically defined to be 0, i.e., $\ACal_{i_1,\dots,i_d} = 0$ if some indices among $i_1,\dots,i_d$ are identical.
\end{definition}
Given a realization of a random hypergraph generated by the symmetric $d$-HSBM, our goal is to exactly recover the underlying communities (i.e., output $\BH^* \BQ$ for some $\BQ \in \Pi_K$) with high probability via a simple iterative procedure.
In view of the fact that problem \eqref{eq: MLE} is reminiscent of the problem formulation of tensor PCA, a natural attempt is to apply a variant of the tensor power iteration \citep{richard2014statistical,huang2022power}. Although many variants of the classic power method have been developed for solving PCA problems with different structural constraints; see, e.g., \citet{journee2010generalized,deshpande2014cone,chen2018projected,zhong2018near,zhu2021orthogonal,wang2021non}, a new variant of the tensor power iteration method has to be developed for problem \eqref{eq: MLE} as it involves a polynomial objective with a matrix variable and binary constraints.

Our approach for tackling problem \eqref{eq: MLE} is to iteratively apply a tensor power step and a projection step that ensures feasibility of the iterate. Specifically, the projected tensor power iteration for tackling problem \eqref{eq: MLE} takes the form
\begin{align}
	\bm{H}^{t+1} \in \TCal \left(\ACal\left[ \left(\bm{H}^{t} \right)^{\otimes (d-1)}\right] \right)~\text{for}~t \geq 1,
\end{align}
where $\mathcal{T}: \mathbb{R}^{n \times K} \rightrightarrows \mathbb{R}^{n \times K}$ represents the projection operator onto $\mathcal{H}$; i.e., for any $\bm{C} \in \mathbb{R}^{n \times K}$,
\begin{align} \label{eq: defprojoperation}
	\mathcal{T}(\bm{C}) \coloneqq {\rm arg} \min \left\{\|\bm{H}-\bm{C}\|_{F}: \bm{H} \in \mathcal{H}\right\},
\end{align}
and the operator $\ACal\left[ \left(\bm{H}^{t} \right)^{\otimes (d-1)}\right]$ is defined in \eqref{eq: grad-tensor-def}.

As an iterative method, the global convergence of PTPM relies on a proper initialization. Specifically, the initial point $\BH^0$ needs to satisfy the partial recovery condition (see, e.g., \citet{dumitriu2021partial}):
\begin{align} \label{eq: init-partial}
	\bm{H}^0 \in \BM_{n,K} \;  \operatorname{s.t.} \min_{\BQ \in \Pi_K } \| \bm{H}^0 - \bm{H}^* \BQ\|_F \leq \theta \sqrt{n},
\end{align}
where $\BM_{n,K}$ represents the set of such $n \times K$ matrices that each row is all zero except for one element being $1$, and $\theta$ is a constant that will be specified later. We remark that condition \eqref{eq: init-partial} can be satisfied by a host of existing initialization methods. For example, \citet{chien2019minimax} and \citet{zhang2022exact} proposed spectral initialization methods, which can obtain an initial point $\bm{H}^0$ satisfying the almost exact recovery condition (see \citet[Definition 4]{abbe2017community}). The almost exact recovery condition is much more stringent than \eqref{eq: init-partial}, and thus these initialization strategies automatically satisfy the partial recovery requirement in \eqref{eq: init-partial}.

Our proposed PTPM for tackling problem \eqref{eq: MLE} is summarized in Algorithm \ref{alg:PGD}. With a qualified initialization $\BH^0$, it first projects $\BH^0$ onto $\HCal$ to guarantee feasibility. Then, it repeatedly refines the estimate by performing $N$-step projected tensor power iterations.

\begin{algorithm}[!htbp]
	\caption{Projected Tensor Power Method for Solving Problem \eqref{eq: MLE}}  
	\begin{algorithmic}[1]  
		\STATE \textbf{Input:} adjacency tensor $\ACal$, positive integer $N$
		\STATE \textbf{Initialize} an $\BH^0$ satisfying \eqref{eq: init-partial} 
		\STATE set $\BH^1 \leftarrow \TCal(\BH^0)$ 		
		\FOR{$t=1,2,\dots,N$}
		\STATE set $\BH^{t+1} \in \TCal \left(\ACal\left[ \left(\bm{H}^{t} \right)^{\otimes (d-1)}\right] \right)$
		\ENDFOR
		\STATE \textbf{Output} $\BH^{N+1}$
	\end{algorithmic}
	\label{alg:PGD}
\end{algorithm}
\begin{remark}\label{rmk:non-uniform}
    Although Algorithm \ref{alg:PGD} is designed for recovering the community of a symmetric $d$-uniform hypergraph SBM, it is applicable to the \emph{non-uniform} hypergraph stochastic block model: Suppose that the observed hypergraph is non-uniform with the size of the hyperedges ranging from 2 to $d_0$. We can introduce a set of dummy nodes indexed by $-3$ to $-d_0$ to reformulate the non-uniform hypergraph as a uniform hypergraph. Specifically, for any hyperedge consisting of nodes $\{i_1,i_2,\dots,i_d \}$ with $d < d_0$, we add the dummy nodes to the hyperedge such that it is of $d_0$ nodes $\{i_1,i_2,\dots,i_d, -(d+1),\dots,-d_0\}$. Then, given a $d_0$-uniform hypergraph, PTPM can be applied to recover the hidden community structure.
\end{remark}
\begin{remark}It is less straightforward for PTPM to recover the community structure from an \emph{asymmetric} hypergraph. Specifically, the asymmetry induces an extra set of assignment matrix variables in the ML estimation problem corresponding to each mode of the observed adjacency tensor. Therefore, one may have to include an extra inner loop to update the set of assignment matrices for each mode at every projected tensor power iteration when applying the idea of PTPM. A related algorithm is the classic high-order orthogonal iteration (HOOI) (see, e.g., \citet[Figure 4.4]{kolda2009tensor}), but it does not involve a projection step.
\end{remark}

\section{Main Results} \label{sec: mainresults}
We first present the main theorem of this paper, which states the exact recovery of PTPM in Algorithm \ref{alg:PGD} down to the information-theoretic limit and the explicit iteration/time complexity of the algorithm.
\begin{theorem}\label{thm: time complexity}
	Let $\ACal \in S^d(\Br^n)$ be an observed adjacency tensor of the random hypergraph generated according to the symmetric $d$-HSBM in Definition \ref{def: dHSBM} with parameters $(n,d, K, p, q)$ and a planted partition $\boldsymbol{H}^* \in \HCal$. Suppose that $p=\alpha \log n/n^{d-1}, q=\beta \log n/n^{d-1}$ with $\frac{(\sqrt{\alpha} - \sqrt{\beta})^2}{K^{d-1}(d-1)!}>1$ and $n$ is sufficiently large. Then, there exists a constant $\gamma>0$, whose value depends only on $\alpha, \beta$, $d$, and $K$, such that the following statement holds with probability at least $1-n^{-\Omega(1)}$: If the initial point satisfies the partial recovery condition in \eqref{eq: init-partial} with
	\begin{equation}
		\theta = \frac{1}{4} \min \left\{\frac{1}{\sqrt{K(d-1)}}, \frac{\gamma K^{d-3/2}}{16(d-1)(\alpha - \beta)} \right\},
	\end{equation}
	then Algorithm \ref{alg:PGD} outputs a true partition in $\left(\lceil 2 \log \log n\rceil+\left\lceil\frac{2 \log n}{\log \log n}\right\rceil+2 \right)$ projected tensor power iterations. Moreover, Algorithm \ref{alg:PGD} outputs a true partition in $\OCal(n \log^2 n/\log \log n)$ time.
\end{theorem}


Let us give the proof outline here before proceeding. Given an initialization $\bm{H}^0$ satisfying the partial recovery condition \eqref{eq: init-partial}, we show that its projection onto $\mathcal{H}$, namely $\bm{H}^1$, still lies in a neighborhood of the ground truth using certain Lipschitz-type inequality (see \eqref{eq: iteration1step} and \eqref{eq: h1bound} for details). Then, by analyzing the effect of the multilinear operation $\mathcal{A}[\cdot]$ and the projection $\mathcal{T}(\cdot)$, we show that each projected tensor power iteration possesses a local contraction property. In other words, the iterate $\bm{H}^t$ gets closer to the ground truth $\bm{H}^*\bm{Q}$ at every iteration (since $\bm{H}^*\bm{Q}$ represents the same partition as $\bm{H}^*$ for any permutation matrix $\bm{Q}$). Specifically, we show that the distance between the iterate $\bm{H}^t$ and the ground truth $\bm{H}^*\bm{Q}$ shrinks by a factor of $\frac{1}{2}$ at each iteration (see \eqref{eq: firstpart_goal}). Moreover, after $N_1 = \lceil 2\log\log n \rceil + 1$ iterations, the iterate gets so close to the ground truth that the contraction factor further reduces to $\mathcal{O}\left(\frac{1}{\sqrt{\log n}}\right)$ (see \eqref{step2:pf-thm-1}). Then, after additional $N_2 = \left\lceil \frac{2\log n}{\log\log n} \right\rceil$ iterations, we can upper bound the distance between the iterate $\bm{H}^{N_1 + N_2}$ and the ground truth $\bm{H}^*\bm{Q}$ strictly by $\sqrt{2}$ (see \eqref{N2:pf-thm-1}). Due to the discrete nature of the feasible set $\mathcal{H}$, this means that the iterate $\bm{H}^{N_1 + N_2}$ has achieved exact recovery (i.e., $\bm{H}^{N_1 + N_2} = \bm{H}^*\bm{Q}$). Multiplying the obtained iteration complexity by the time complexity of each projected tensor power iteration, we can derive the total time complexity of PTPM in Theorem \ref{thm: time complexity}.

In the remaining part of this section, we provide the proof of Theorem \ref{thm: time complexity}. We break down the analysis of each projected tensor power iteration by studying the effect of the multilinear operation $\mathcal{A}[\cdot]$ (in Lemmas \ref{eq: lemma-multi-linear}, \ref{prop: concentration}, \ref{lemma: difbinom}, and \ref{lemma: ground truth}) and the projection $\mathcal{T}(\cdot)$ (in Lemma \ref{lemma: lipproj}), respectively, and then establish a local contraction property of the projected tensor power iteration in Proposition \ref{prop: localcontraction}. Based on the contraction property, we derive the iteration complexity of PTPM to achieve exact community recovery in Theorem \ref{thm: iteration complexity} and also the time complexity in Theorem \ref{thm: time complexity}.

We start with characterizing the effect of the multilinear operation $\ACal[\BH^{\otimes(d-1)}]$ in Lemma \ref{eq: lemma-multi-linear}. To make the presentation more concise, we slightly change the definition of $\ACal$ in this lemma: For the diagonal elements with some of the indices $i_1,\ldots,i_d$ being identical, we have
$\ACal_{i_1,\dots,i_d} \sim \mathbf{Bern}(p)$ if the $d$ nodes are in the same community and $\ACal_{i_1,\dots,i_d} \sim \mathbf{Bern}(q)$ otherwise. In fact, the error term incurred by this modification is negligible, on which we will comment in Remark \ref{remark: diagonalofA}.
\begin{lemma} \label{eq: lemma-multi-linear}
	Suppose that $\varepsilon \in \left(0,1/\sqrt{K(d-1)}\right) $ and $\bm{H} \in \HCal$ such that $\| \bm{H} - \bm{H}^*\BQ \|_F \leq \varepsilon \sqrt{n}$ for some $\BQ \in \Pi_K$. Then, with probability at least $1-n^{-10}$, we have
	\begin{align} \label{eq: multi-linear}
		& \quad \left\| \ACal\left[\bm{H}^{\otimes (d-1)}\right] - \ACal\left[ (\bm{H}^*\BQ)^{\otimes (d-1)}\right] \right\|_F \nonumber \\ 
		& \leq \left(\frac{4(d-1)m^{d-2}\varepsilon n}{\sqrt{K}} (p-q) + C\sqrt{\log n} \right) \|\bm{H} - \bm{H}^*\BQ\|_F
	\end{align}  
	with $C>0$ being a constant.
\end{lemma}

	The idea of proving Lemma~\ref{eq: lemma-multi-linear} is to separately bound
\[
\left\| (\EE[\ACal])\left[\bm{H}^{\otimes (d-1)}\right] - (\EE[\ACal])\left[ (\bm{H}^*\BQ)^{\otimes (d-1)}\right] \right\|_F
\]
and 
\[
\left\| \Delta\left[\bm{H}^{\otimes (d-1)}\right] - \Delta\left[ (\bm{H}^*\BQ)^{\otimes (d-1)}\right] \right\|_F
\]
with high probability, where $\Delta = \ACal - \EE [\ACal]$ is the deviation of $\ACal$ from its expectation. On one hand, the former term can be computed via the definition of $\ACal$ and some algebraic inequalities; on the other hand, using \eqref{eq: grad-two-eq-expression}, the latter term would yield
\begin{equation} \label{eq:bdn}
    \left \| \mathcal{M}(\Delta) \left(\bm{H}_k^{\odot (d-1)} - (\bm{H}^*\bm{Q})_k^{\odot (d-1)} \right) \right\|_2~\text{for}~ k\in [K];
\end{equation}
see \eqref{eq: boundnoise}. A natural attempt to bound \eqref{eq:bdn} is to directly apply the inequalities \eqref{eq: kroneck-ineq} and the results in \citet[Theorem 2.3; Remark 2.2; Lemma 6.1]{zhou2021sparse}. Yet, this approach would yield a loose upper bound. As a remedy, we present a useful concentration bound of \eqref{eq:bdn}, as shown in Lemma \ref{prop: concentration}. The trick is to leverage the Kahn-Szemer{\'e}di argument \citep{feige2005spectral}, which has been applied to obtain bounds for the spectral norm of sparse binary random square matrices \citep{lei2015consistency} and tensors \citep{zhou2021sparse}. The result is new and can be of interest to other applications with high-order relations.

\begin{lemma} \label{prop: concentration}
	Let $d \geq 2$ and $\frac{1}{s_1} \leq \ell \leq n$ for some constant $s_1\geq 1$. Let $\BA \in \Br^{n \times n^{d-1}}$ be a random matrix whose $(i,j)$-th entry independently follows $\mathbf{Bern}(p_{ij})$. Set $\BP = \EE[\BA]$ and assume that $\xi \coloneqq n^{d-1} \cdot p_{\max} \geq c_0 \log n$ for $p_{\max} \coloneqq \max_{i\in [n],j\in[ n^{d-1}]}p_{ij}$ and some constant $c_0 >0$. Given a vector $\by \in \Br^{n^{d-1}}$ that has exactly $\ell \cdot n^{d-2}$ nonzero elements with half of them taking $1$ and the others taking $-1$, then, for any $r >0$, there exists a constant $C>0$ such that 
	\begin{equation}\label{eq:concentration}
		\| (\BA - \BP) \by \|_2 \leq C\sqrt{\xi}\cdot \sqrt{\ell}
	\end{equation}
	with probability at least $1-n^{-r}$.
\end{lemma}

\begin{remark} \label{remark: diagonalofA}
The results in Lemma \ref{eq: lemma-multi-linear} still hold when the diagonal entries of $\mathcal{A}$ are 0 (as defined in Definition \ref{def: dHSBM}). In fact, since the modification on the tensor $\mathcal{A}$ only takes place in the diagonal entries of the tensor, it is negligible when compared with the effect of $\mathbb{E}[\mathcal{A}]$ acting on the matrices $\bm{H}$ and $\bm{H}^*\bm{Q}$. We will leave a detailed discussion in the appendix.
\end{remark}

Next, we shift our attention to the projection operator $\TCal$ and present its Lipschitz-like property in the following lemma \citep[Lemma 3]{wang2021optimal}.
\begin{lemma} \label{lemma: lipproj}
	Let $\delta>0, \bm{C} \in \mathbb{R}^{n \times K}$ be arbitrary and $m=n / K$. If there exists a collection of index sets $\mathcal{I}_{1}, \ldots, \mathcal{I}_{K}$ satisfying $\cup_{k=1}^{K} \mathcal{I}_{k}=[n], \mathcal{I}_{k} \cap \mathcal{I}_{\ell}=\emptyset$, and $\left|\mathcal{I}_{k}\right|=m$ such that $\bm{C}$ satisfies
	\begin{equation} \label{eq: Centeydelta}
		C_{i k}-C_{i \ell} \geq \delta
	\end{equation}
	for all $i \in \mathcal{I}_{k}$ and $1 \leq k \neq \ell \leq K$. Then, for any $\bm{V} \in$ $\mathcal{T}(\bm{C}),~ \bm{C}^{\prime} \in \mathbb{R}^{n \times K}$ and $\bm{V}^{\prime} \in \mathcal{T}\left(\bm{C}^{\prime}\right)$, we have
	\begin{equation}	\left\|\bm{V}-\bm{V}^{\prime}\right\|_{F} \leq \frac{2\left\|\bm{C}-\bm{C}^{\prime}\right\|_{F}}{\delta} .
    \end{equation}
\end{lemma}
Considering $\bm{C} = \mathcal{A}\left[(\bm{H}^*)^{\otimes (d-1)}\right]$, the following two lemmas show that condition \eqref{eq: Centeydelta} in Lemma \ref{lemma: lipproj} can be satisfied with high probability. We start with a lemma addressing a binomial tail inequality below.

\begin{lemma}\label{lemma: difbinom}
	Let $m=n / K$ and $\alpha>\beta>0$ be constants. Suppose that $\left\{W_{i}\right\}_{i=1}^{\binom{m-1}{d-1}}$ are i.i.d. $\mathbf{Bern}(\alpha \log n / n^{d-1})$ and $\left\{Z_{i}\right\}_{i=1}^{\binom{m}{d-1}}$ are i.i.d. $\mathbf{Bern}(\beta \log n / n^{d-1})$ that is independent of $\left\{W_{i}\right\}_{i=1}^{\binom{m-1}{d-1}}$. Then, for any $\gamma \in \mathbb{R}$, it holds that
	\begin{align}
		& \quad\, {\rm Pr}\left[\sum_{i=1}^{\binom{m-1}{d-1}} W_{i}-\sum_{i=1}^{\binom{m}{d-1}} Z_{i} \leq \frac{\gamma}{(d-1)!} \log n\right] \nonumber \\
		& \leq n^{-\frac{ \binom{m-1}{d-1}(\sqrt{\alpha}-\sqrt{\beta})^{2}}{n^{d-1}}+\frac{\gamma \log (\alpha / \beta)}{2 (d-1)!} + \frac{\left( \binom{m}{d-1} - \binom{m-1}{d-1}\right) (\sqrt{\alpha \beta} - \beta)}{n^{d-1}} }.
	\end{align}
\end{lemma}
Lemma \ref{lemma: difbinom} can be proved by adapting the proof techniques in \citet[Lemma 8]{abbe2020entrywise}. This turns out to be the key to proving the information-theoretic optimality of our algorithm. Recalling that $\bm{C} = \mathcal{A}\left[(\bm{H}^*)^{\otimes (d-1)}\right]$, by the definitions of the adjacency tensor $\mathcal{A}$ and the multilinear operator $\mathcal{A}[\cdot]$, one can compute
\[
C_{ik} \overset{\textnormal{d}}{=} (d-1)! \cdot \sum_{i=1}^{\binom{m-1}{d-1}} W_i, \; i \in \ICal_k=\{i\in [n]: H_{ik}^*=1 \}
\]
and
\[
C_{i\ell} \overset{\textnormal{d}}{=} (d-1)! \cdot \sum_{i=1}^{\binom{m}{d-1}} Z_i,\; l\neq k,
\]
where $\{W_i\}$ are i.i.d. $\mathbf{Bern}(\alpha \log n / n^{d-1})$, and $\left\{Z_{i}\right\}$ are i.i.d. $\mathbf{Bern}(\beta \log n / n^{d-1})$ and independent of $\{W_i\}$. Utilizing the result of Lemma \ref{lemma: difbinom}, we then have the following lemma.

\begin{lemma}\label{lemma: ground truth}
	Let $\alpha > \beta >0$ be constants. Denote $\bm{C} = \ACal\left[(\bm{H}^*)^{\otimes (d-1)}\right]$ and $\ICal_k =\{i\in [n]: H_{ik}^*=1 \}$ for all $k\in[K]$. If $\frac{(\sqrt{\alpha} - \sqrt{\beta})^2}{K^{d-1}(d-1)!} >1$ and $n$ is sufficiently large, then there exists a constant $\gamma >0$, which depends only on $\alpha,\beta,d$, and $K$, such that for all $i \in \ICal_k$ and $1 \leq k \neq \ell \leq K$,
	\begin{align} \label{eq: cik-cil}
		C_{ik}-C_{i\ell} \geq \gamma \log n
	\end{align}
	holds with probability at least $1-n^{-\Omega(1)}$.
\end{lemma}

Collecting the results from all the lemmas above, we can now show that the projected tensor power iteration possesses a contraction property in a certain neighborhood of $\BH^* \BQ$ for some $\BQ \in \Pi_K$. 
\begin{proposition}	\label{prop: localcontraction}
	Let $\alpha > \beta >0$ be constants satisfying $(\sqrt{\alpha} - \sqrt{\beta})^2 > K^{d-1}(d-1)!$. Suppose that $n > \exp(16C^2/\gamma^2) \}$ and $n$ is sufficiently large such that \eqref{eq: cik-cil} holds.
	Then, with probability at least $1-n^{-\Omega(1)}$, for any fixed $\BH \in \mathcal{H}$ and $\varepsilon \in \left(0, \min \left\{\frac{1}{\sqrt{K(d-1)}}, \frac{\gamma K^{d-3/2}}{16(d-1)(\alpha - \beta)} \right\} \right)$ such that 
	\begin{align}
		\| \BH - \BH^* \BQ\|_F \leq \varepsilon \sqrt{n}
	\end{align}
	for some $\BQ \in \Pi_K$, it holds that
	\begin{align}
		\| \bm{V} - \BH^* \BQ \|_F \leq \kappa \|\BH - \BH^* \BQ\|_F
	\end{align}
	for any $\bm{V} \in \TCal\left( \ACal\left[\BH^{\otimes (d-1)}\right] \right)$, where $$\kappa = 4 \max \left\{\frac{4(d-1)\varepsilon (\alpha - \beta)}{\gamma K^{d-3/2}},  \frac{C}{\gamma \sqrt{\log n}} \right\} \in (0,1).$$
\end{proposition}

\begin{figure*}[t]
\centering
	\begin{minipage}[b]{0.31\linewidth}
		\centering
		\centerline{\includegraphics[width=\linewidth]{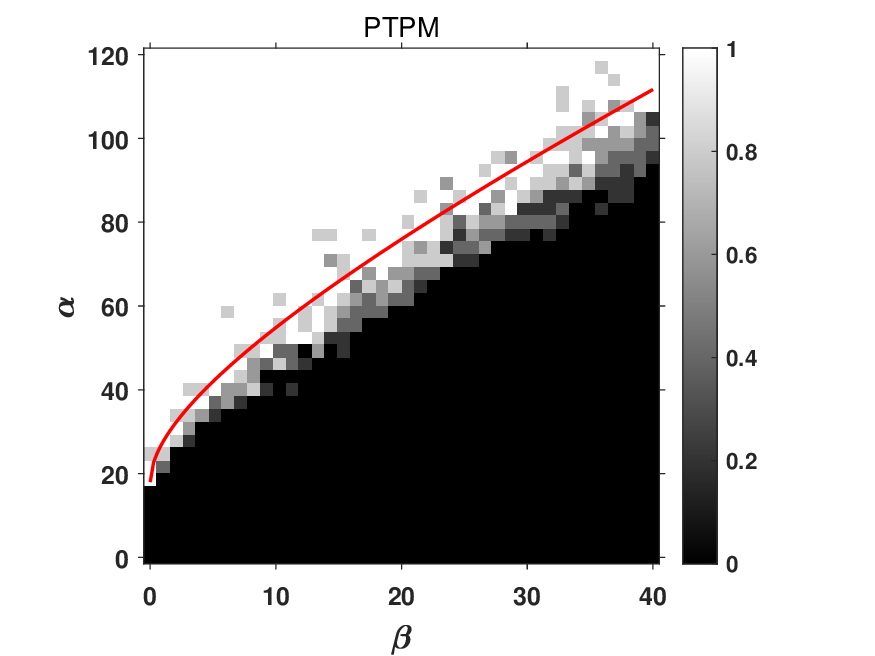}}
		\centerline{(a) PTPM}\medskip
	\end{minipage}
	\begin{minipage}[b]{0.31\linewidth}
		\centering
		\centerline{\includegraphics[width=\linewidth]{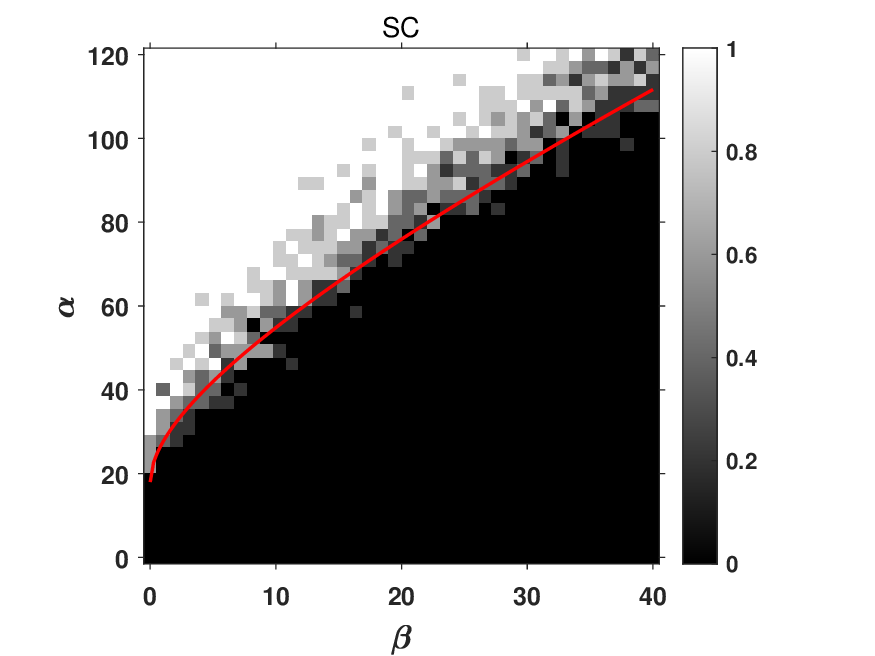}}
		\centerline{(b) SC}\medskip
	\end{minipage}
	\begin{minipage}[b]{0.31\linewidth}
		\centering
		\centerline{\includegraphics[width=\linewidth]{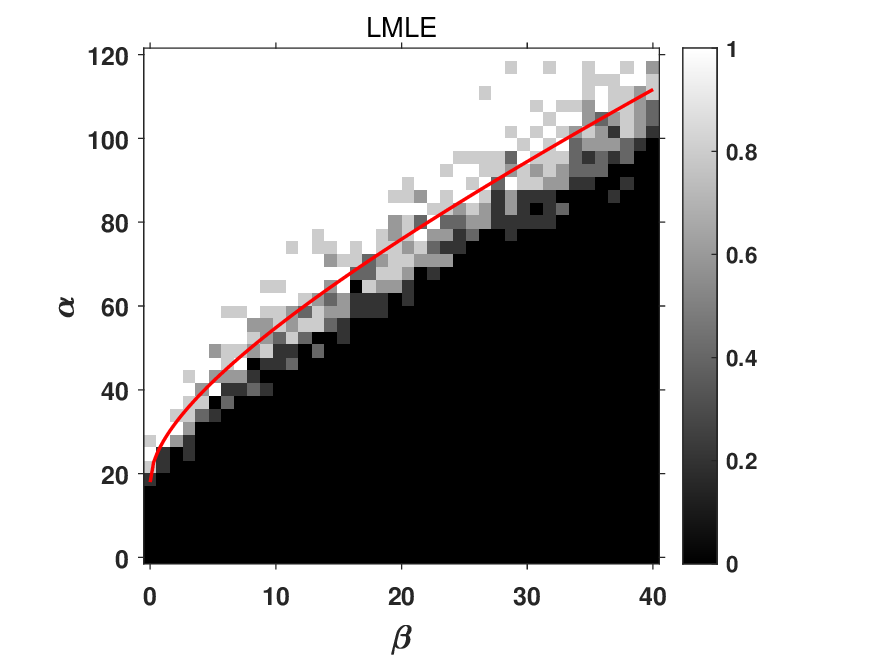}}
		\centerline{(c) LMLE}\medskip
	\end{minipage}
	\vskip -0.2in
	\caption{Phase transition in hypergraphs generated by $3$-HSBM in the setting of $n=210, K=3$: The $x$-axis is $\beta$, the $y$-axis is $\alpha$, and darker pixels represent lower empirical probability of success. The red curve is the information-theoretic threshold $\sqrt{\alpha}-\sqrt{\beta}=\sqrt{18}$.}
	\label{fig: phase-transition}
\end{figure*}

\begin{figure*}[t]
	\begin{center}
		\begin{minipage}[b]{0.31\linewidth}
			\centering
			\centerline{\includegraphics[width=\linewidth]{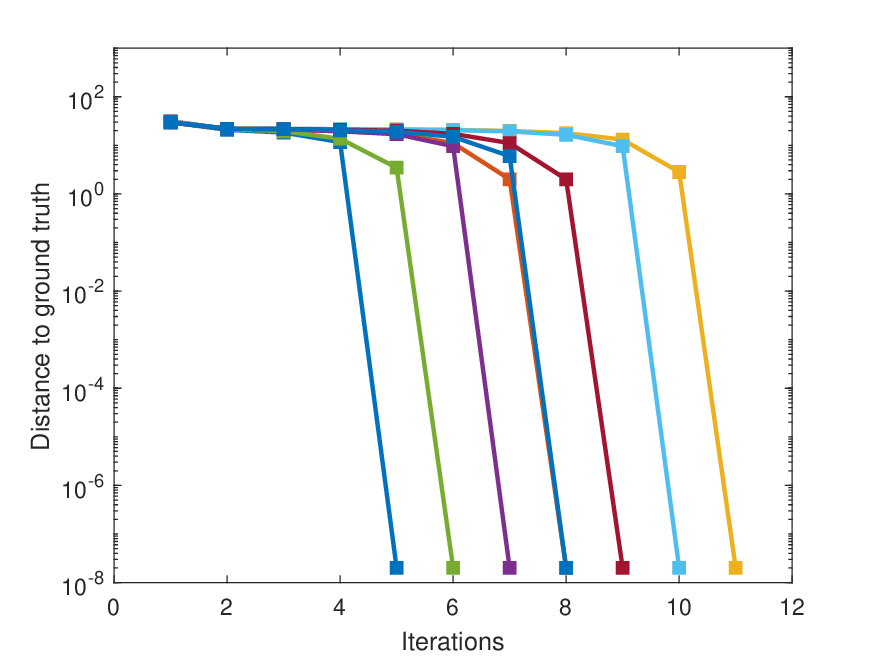}}
			\centerline{(a) {\footnotesize $(d,K,\alpha,\beta)=(3,2,33,8)$}}\medskip
		\end{minipage}
		\begin{minipage}[b]{0.31\linewidth}
			\centering
			\centerline{\includegraphics[width=\linewidth]{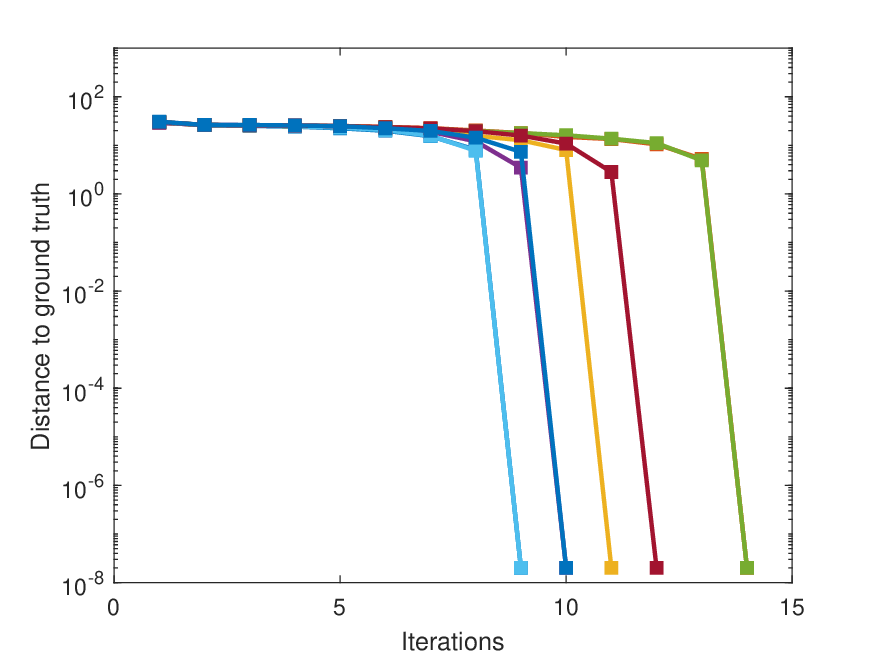}}
			\centerline{(b) {\footnotesize $(d,K,\alpha,\beta)=(3,4,130,32)$}} \medskip
		\end{minipage}
		\begin{minipage}[b]{0.3\linewidth}
			\centering
			\centerline{\includegraphics[width=\linewidth]{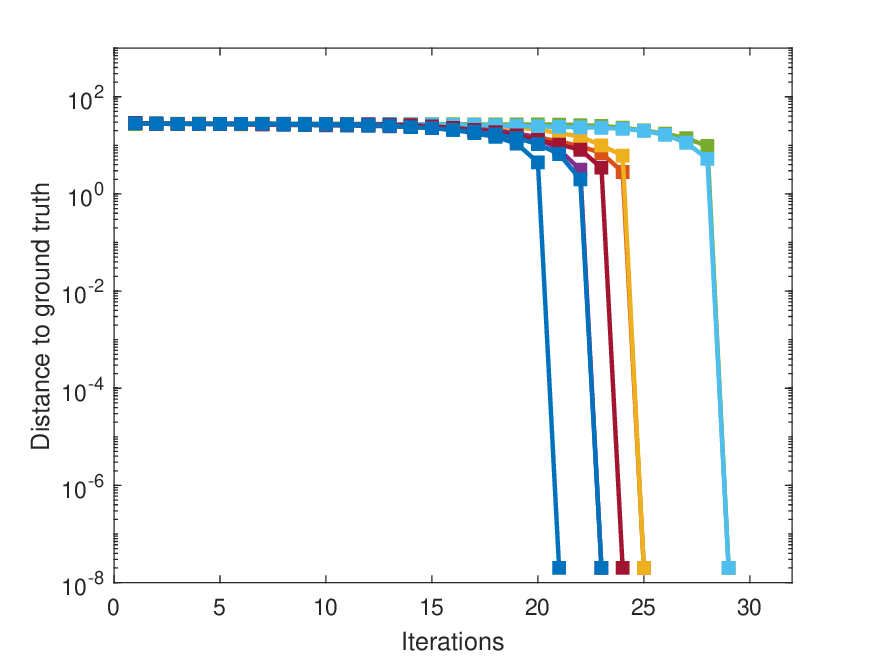}}
			\centerline{(c) {\footnotesize $(d,K,\alpha,\beta)=(3,8,400,64)$}} \medskip
		\end{minipage}
	\end{center}
	\vskip -0.2in
	\caption{Convergence performance of PTPM with random initialization: The $x$-axis is the number of iterations, and the $y$-axis is the distance from an iterate to the ground truth, i.e., $\min_{\BQ \in \Pi_K}\|\BH^t-\BH^*\BQ\|_F$, where $\BH^t$ is the $t$-th iterate generated by PTPM.}
	\label{fig: convergence}
\end{figure*}
Armed with the result of the local contraction property, we are ready to derive the iteration complexity of PTPM.
\begin{theorem}\label{thm: iteration complexity}
	Denote $\phi = \frac{C K^{d-3/2}}{64(d-1)(\alpha - \beta)}$ and $\theta = \frac{1}{4}\min \left\{\frac{1}{\sqrt{K(d-1)}}, \frac{\gamma K^{d-3/2}}{16(d-1)(\alpha - \beta)} \right\}$. Suppose that $n > \exp \left(\max \left\{ \frac{64C^2}{\gamma^2}, \frac{\gamma^2}{C^2}, \frac{\phi^2}{\theta^2}, 2\phi^2,\frac{256C^4}{\gamma^4} \right\}\right)$ and $n$ is sufficiently large such that \eqref{eq: cik-cil} holds. Then, the following statement holds with probability at least $1 - n^{-\Omega(1)}$: If the initial point $\BH^0$ satisfies \eqref{eq: init-partial}, then Algorithm \ref{alg:PGD} outputs $\BH^* \BQ$ within $\left(\lceil 2\log\log n \rceil+\left\lceil \frac{2\log n}{\log\log n} \right\rceil+2\right)$ projected tensor power iterations.
\end{theorem}
Recall that the matrix $\bm{H}^*\bm{Q}$ represents the same partition as $\bm{H}^*$ for any permutation matrix $\bm{Q}$. Hence, outputting $\bm{H}^*\bm{Q}$ in Theorem \ref{thm: iteration complexity} would imply achieving exact recovery in Theorem \ref{thm: time complexity}. Further multiplying the derived iteration complexity by the time complexity of each projected tensor power iteration, we can obtain the total time complexity of PTPM in Theorem \ref{thm: time complexity}.

\section{Experiment Results} \label{sec: experiments}
In this section, we report the recovery performance and numerical efficiency of our proposed PTPM for recovering communities on synthetic/real data. We also compare our method with two existing methods, which are the spectral clustering (SC) method in \citet{ghoshdastidar2015provable} and the local MLE (LMLE) method in \citet{chien2019minimax}. In the implementation, we employ \citet[Algorithm 2]{chien2019minimax} for computing an initial point $\BH^0$ in Algorithm \ref{alg:PGD} if we do not specify the initialization method. Moreover, to possibly reduce computational time, we implement a simplified version of the LMLE method as stated in \citet[Remark 4.2]{chien2019minimax}. The MATLAB function \textsf{eigs} for computing the eigenvectors is applied in the SC method and the first stage of the LMLE method. The MATLAB function \textsf{kmeans} for computing the partition is used in the SC method. We use the Tensor Toolbox \citep{osti_1349514} to perform tensor operations and compute $\ACal\left[ \left(\bm{H}^{t} \right)^{\otimes (d-1)}\right]$ based on \eqref{eq: grad-two-eq-expression}. 
\subsection{Phase Transition and Computational Time}
We first conduct experiments to examine the phase transition property and the running time of the aforementioned methods for exact recovery under the symmetric $3$-HSBM. We set $n=210,K=3$ in the experiments and let the parameter $\alpha$ vary from $0$ to $120$ with increments of $3$ and the parameter $\beta$ vary from $0$ to $40$ with increments of $1$. For each pair of $\alpha$ and $\beta$, we generate $5$ instances and calculate the success ratio of exact recovery for all the tested methods. The phase transition results are shown in Figure \ref{fig: phase-transition}. It can be observed that all the methods exhibit a phase transition phenomenon. Moreover, Figure \ref{fig: phase-transition}(a) indicates that PTPM achieves the optimal recovery threshold, which complements our theoretical findings. Besides, we record the total CPU time consumed by each method for completing the phase transition
experiments with different parameters in Table \ref{tab: cputime}. It can be observed that PTPM is faster than SC and substantially faster than LMLE.

\begin{table}[htbp]
	\centering
		\caption{Total CPU time (in seconds) of the different methods in the phase transition experiments.}
  \vskip 0.15 in
	\begin{tabular}{cccc}
		\hline Time (s) & PTPM & SC & LMLE \\
		\hline$n=210, K=3,d=3$ & $\mathbf{76}$ & 246 & $>500$ \\
        \hline$n=210, K=6,d=3$ & $\mathbf{157}$ & 451 & $>1000$ \\
    \hline$n=420, K=3,d=3$ & $\textbf{189}$ & 321 & $>1000$ \\
      \hline$n=420, K=6,d=3$ & $\textbf{330}$ & 628 & $>1000$ \\
		\hline
	\end{tabular}
	\label{tab: cputime}
\end{table}

\subsection{Convergence Performance}
Next, we test the convergence performance of PTPM for exact community recovery under the symmetric $d$-HSBM. Specifically, we choose three different sets of $(d,K, \alpha,\beta)$ such that $\frac{(\sqrt{\alpha} - \sqrt{\beta})^2}{K^{d-1}(d-1)!} >1$ and generate hypergraphs with $n=480$ nodes. In addition, the initial point $\BH^0$ in Algorithm \ref{alg:PGD} is generated via $\BH^0 \in \TCal(\bm{G})$, where $\bm{G} \in \mathbb{R}^{n\times K}$ is a random Gaussian matrix. In each hypergraph realization, we run PTPM 8 times with different random initial points and then plot the distances of the iterative points to the ground truth, i.e., $\min_{\BQ \in \Pi_K}\|\BH^t-\BH^*\BQ\|_F$, against the iteration number in Figure \ref{fig: convergence}. We see that PTPM achieves exact community recovery within 30 iterations even with random initialization. This demonstrates the power of PTPM in the ML estimation problem of the symmetric $d$-HSBM.

We also run a few experiments to assess the performance of PTPM when applying to a \emph{non-uniform} hypergraph with 400 nodes ($n=400$), where the 2-hyperedges and 3-hyperedges are generated independently via the graph SBM and the 3-HSBM with two underlying communities ($K=2$), respectively. We follow the procedure mentioned in Remark~\ref{rmk:non-uniform} to reformulate each hypergraph as a uniform hypergraph and run the experiment 6 times with random initialization. The results show that exact community recovery could be achieved within 10 iterations, which demonstrates the power of PTPM when it is applied to non-uniform hypergraphs. 

\subsection{Recovery Accuracy and Efficiency}

To evaluate the performance of PTPM, we also conduct a real data experiment shown as follows. Based on the 1984 US Congressional voting records available at the UCI repository\footnote{http://archive.ics.uci.edu/ml}, we choose two balanced communities ($K=2$) with $n=336$ number of Congressmen and randomly generate a 3-way symmetric adjacency tensor according to their votes on four issues (namely, columns 4, 5, 12, and 15 in the record). Specifically, if the three Congressmen indicate the same stance on a specific issue, we generate a 3-uniform hyperedge with a probability of $0.05$ and no hyperedge otherwise; cf. \citet{ghoshdastidar2017consistency,chien2019minimax}. To test the performance of PTPM, we generate an initial point $\bm{H}^0\in\mathcal{T}(\bm{G})$ for some randomly generated Gaussian matrix $\bm{G}$ and terminate PTPM when the iteration number reaches 20. We run each algorithm 10 times and select the solution with the lowest function value as its recovered solution. The misclassification rate and running time are recorded in Table \ref{tab: realdata}. As can be seen, our proposed method has better accuracy and efficiency than existing methods.

\begin{table}[htbp]
	\centering
		\caption{The misclassification rate and the total CPU times (in seconds) of the methods on the UCI dataset.}
  \vskip 0.15in
	\begin{tabular}{cccc}
		\hline Algorithms & PTPM & SC & LMLE \\
		\hline Misclassification rate & $\mathbf{0.07}$ & 0.08 & 0.10 \\
        \hline Time (s) & $\mathbf{0.85}$ & 1.56 & $>10$ \\
		\hline
	\end{tabular}
	\label{tab: realdata}
\end{table}

\section{Conclusion} \label{sec: conclusion}
In this work, we developed a simple method, namely PTPM, for tackling the ML estimation problem of the symmetric $d$-HSBM. Our theory guarantees that, given an initialization satisfying the mild partial recovery condition, PTPM achieves exact community recovery down to the information-theoretic limit and has a low time complexity of $\OCal(n\log^2n/\log\log n)$. One intriguing future direction is to design a simple iterative method (e.g., \citet[Algorithm 1]{wang2021non}) for obtaining a qualified initial point. Another future direction is to apply our method to other complex relational models with underlying community structures, such as multilayer networks \citep{jing2021community,fan2022alma,chen2022global} and multiplex networks \citep{pensky2021clustering,noroozi2022sparse}.
\section{Acknowledgments}
The authors thank Prof. Yizhe Zhu, Dr. Yongfeng Li, and Prof. Xiongjun Zhang for their helpful discussions. This research is supported in part by the Hong Kong Research Grants Council (RGC) General Research Fund (GRF) project CUHK 14205421.

\bibliographystyle{icml2023}
\bibliography{HSBM.bib}

\onecolumn
\begin{appendix}
	\begin{center}
		{\Large \bf Appendix\\ }
  	\vspace{0.4cm}
		\leftline{In the appendix, we provide proofs of the technical results presented in Section \ref{sec: mainresults}.}
  \section{Proof of Theorem \ref{thm: time complexity}}
    \begin{proof}
			From Theorem \ref{thm: iteration complexity}, we know that Algorithm \ref{alg:PGD} yields a true partition in $\Big(\lceil 2 \log \log n\rceil+$ $\left\lceil\frac{2 \log n}{\log \log n}\right\rceil+2\Big)$ projected tensor power iterations with probability at least $1-n^{-\Omega(1)}$. Besides, the time complexity of performing the projection operation $\TCal$ (i.e., solving \eqref{eq: defprojoperation}) is $\OCal(K^2 n \log n)$ \citep[Proposition 1]{wang2021optimal}. It remains to show that the time complexity of computing $\ACal\left[ \left(\bm{H}^{t} \right)^{\otimes (d-1)}\right]$ (or equivalently $ \MCal(\ACal) \left( (\bm{H}^t)^{\odot (d-1)} \right)$ by \eqref{eq: grad-two-eq-expression}) is $\OCal(n \log n)$ with high probability.
			Since $\mathcal{A}$ is generated according to the symmetric $d$-HSBM with $p=\alpha \log n/ n^{d-1}$ and $q = \beta \log n/ n^{d-1}$, by denoting the number of nonzero entries of $\ACal$ as $\| \ACal\|_0$, we have 
			\begin{equation}
				\|\ACal\|_0 \overset{\textnormal{d}}{=} d! \cdot \sum_{i=1}^{ K \tbinom{m}{d} } W_i + d! \cdot \sum_{i=1}^{\tbinom{n}{d} -K \tbinom{m}{d}} Z_i,
			\end{equation}
			where $\{W_i\}_{i=1}^{ K \tbinom{m}{d}}$ are i.i.d. $\mathbf{Bern}(p)$ and $\{Z_i\}_{i=1}^{\tbinom{n}{d} -K \tbinom{m}{d}}$ are i.i.d. $\mathbf{Bern}(q)$, independent of $\{W_i\}_{i=1}^{K\tbinom{m}{d}}$. The expectation and variance of $\|\ACal\|_0$ are given by
			\begin{align*}
				\mathbb{E}[\|\ACal\|_0] & = d! \cdot \left(K\tbinom{m}{d}p + \left(\tbinom{n}{d} -K \tbinom{m}{d}\right)q \right),\nonumber\\
				{\rm Var}\left[\|\ACal\|_0 \right] &= (d!)^2 \cdot \left(K\tbinom{m}{d}p(1-p) + \left(\tbinom{n}{d} -K \tbinom{m}{d}\right)q(1-q) \right)\\
				& \leq (d!)^2 \cdot \left(K\tbinom{m}{d}p + \left(\tbinom{n}{d} -K \tbinom{m}{d}\right)q \right).
			\end{align*}
			Applying Bernstein's inequality \citep[Theorem 2.8.4]{vershynin2018high} gives
			\begin{align*}
				&\quad \operatorname{Pr} \left[\left|\|\ACal\|_0 - d! \cdot \left(K\tbinom{m}{d}p + \left(\tbinom{n}{d} -K \tbinom{m}{d}\right)q \right) \right|\ge 2d! \cdot \left(K\tbinom{m}{d}p + \left(\tbinom{n}{d} -K \tbinom{m}{d}\right)q \right)\right]\\
				& = \operatorname{Pr} \left[ \left| \frac{\|\ACal\|_0}{d!}- \left(K\tbinom{m}{d}p + \left(\tbinom{n}{d} -K \tbinom{m}{d}\right)q \right) \right|\ge 2\left(K\tbinom{m}{d}p + \left(\tbinom{n}{d} -K \tbinom{m}{d}\right)q \right)\right] \\
				&\le 2 \exp\left(-\frac{2\left(K\tbinom{m}{d}p + \left(\tbinom{n}{d} -K \tbinom{m}{d}\right)q \right)^2}{\left(K\tbinom{m}{d}p + \left(\tbinom{n}{d} -K \tbinom{m}{d}\right)q \right) + 2\left(K\tbinom{m}{d}p + \left(\tbinom{n}{d} -K \tbinom{m}{d}\right)q \right)/3}\right)\\
				&\le 2\exp \left( -\left(K\tbinom{m}{d}p + \left(\tbinom{n}{d} -K \tbinom{m}{d}\right)q \right)\right)\\
				& \leq 2 \exp \left(- \left(\frac{K m^d}{2^{d-1} d!}p + \frac{n^d }{2^{d-1} d!}q - K\frac{m^d}{d!}q \right) \right)\\
				&= 2n^{- \left(\frac{\beta}{2^{d-1}d!} + \frac{K\alpha}{2^{d-1}K^d d!} - \frac{K\beta}{K^d d!} \right)n} \leq 2n^{- \left( \frac{K\alpha}{2^{d-1}K^d d!} \right)n},
			\end{align*}
			where we assume $m-d+1 \geq m/2$ in the third inequality and the last inequality is due to $K\geq 2$. It holds with probability at least $1 - 2n^{- \left( \frac{K\alpha}{2^{d-1}K^d d!} \right)n}$ that the number of nonzero entries in $\mathcal{A}$ is less than $3 d! \cdot \left(K\tbinom{m}{d}p + \left(\tbinom{n}{d} -K \tbinom{m}{d}\right)q \right) \leq  \frac{3\alpha + 3K^{d-1}\beta}{K^{d-1}}n \log n$. Thus, the time complexity of computing $\ACal\left[ \left(\bm{H}^{t} \right)^{\otimes (d-1)}\right]$ is $\frac{3\alpha + 3K^{d-1}\beta}{K^{d-1}}n \log n$ with probability at least $1 - 2n^{- \left( \frac{K\alpha}{2^{d-1}K^d d!} \right)n}$. The desired bound on the total time complexity of Algorithm \ref{alg:PGD} then follows.
		\end{proof}
		\section{Proof of Lemma \ref{eq: lemma-multi-linear}}
		\begin{proof}
			Let $\Delta = \ACal - \BE[\ACal]$. Without loss of generality, we assume that $\bm{H}^* = \bm{I}_K \otimes \b1_m \in \mathbb{R}^{n \times K}$, which implies that
			\begin{equation*}
				\left(\BH^* \right)^{\otimes d} = \left(\bm{I}_K ^{\otimes d}\right) \otimes \left(\b1_m^{\otimes d}\right) \in S^d(\Br^n);
			\end{equation*}
			see \citet[Section 2.1]{cichocki2016tensor} for the Kronecker product of tensors. 
			By the definition of $\ACal$, one can verify that
			\begin{align} \label{eq: eadecom}
				\BE[\ACal] 
				=  \BCal \otimes \left(\b1_m^{\otimes d} \right) = q\cdot \left(\b1_n^{\otimes d}\right) + (p-q)\cdot \left(\BH^* \right)^{\otimes d},
			\end{align}
			where $\BCal \in S^d(\Br^K)$ is such that $\BCal_{k_1,\dots,k_d}$ equals to $p$ if $k_1=\dots=k_d$ and $q$ otherwise. Let us decompose $(\bm{H}_j)^{\otimes(d-1)}, j\in [K]$ into several orthogonal parts:
			\begin{align} \label{eq: Horthdecom}
				(\bm{H}_j)^{\otimes(d-1)} = \sum_{k=1}^K ((\BH^* \BQ)_k)^{\otimes (d-1)} Z_{kj} + \GCal_j.
			\end{align}
			Here, $\bm{Z} \in \Br^{K \times K}$ is given by
			\begin{equation}\label{eq: Zexpression}
				Z_{kj} = \frac{1}{m^{d-1}} \left \langle (\bm{H}_j)^{\otimes(d-1)}, ((\BH^* \BQ)_k)^{\otimes (d-1)} \right\rangle = \frac{1}{m^{d-1}} \left ( \langle \BH_j, (\BH^* \BQ)_k \rangle \right)^{d-1}
			\end{equation}
   and $\GCal_j \in S^{d-1}(\Br^n)$ satisfies $\left\langle \GCal_j,((\BH^* \BQ)_k)^{\otimes (d-1)} \right\rangle = 0$ for $k\in [K]$. 
			It follows from the definition in \eqref{eq: grad-tensor-def} that for $i\in [n], k\in [K]$,
			\begin{align}
				\left( \left(\b1_n^{\otimes d} \right)\left[\bm{H}^{\otimes (d-1)} \right] \right)_{ik} = \left(\b1_n^\top \bm{H} _k \right)^{d-1}= m^{d-1} = \left( \left(\b1_n^{\otimes d} \right)\left[(\BH^* \BQ)^{\otimes (d-1)} \right] \right)_{ik}. \label{linearoperator-zero}
			\end{align}
			In addition, it is implied by \eqref{eq: grad-tensor-def} that
			\begin{align}
				& \quad \, \left( \left(\left(\BH^*\right)^{\otimes d} \right)\left[\bm{H}^{\otimes (d-1)}\right] \right)_{ik} \nonumber \\
				& =  \left(\left( \left(\bm{I}_K \otimes \b1_m \right)^{\otimes d} \right)\left[\bm{H}^{\otimes (d-1)} \right] \right)_{ik}  \nonumber \\
				& = \sum_{1 \leq i_2,\dots,i_d \leq n} \left( \left(\bm{I}_K \otimes \b1_m \right)^{\otimes d} \right)_{i,i_2,\dots,i_d} \times \bm{H}_{i_2k} \times \cdots \times \bm{H}_{i_dk} \nonumber \\
				& = \sum_{1 \leq i_2,\dots,i_d \leq n} \sum_{j=1}^K \left(\bm{I}_K \otimes \b1_m \right)_{ij} \times \left(\bm{I}_K \otimes \b1_m \right)_{i_2j} \times \dots \times \left(\bm{I}_K \otimes \b1_m \right)_{i_dj}  \times \bm{H}_{i_2k} \times \cdots \times \bm{H}_{i_dk}    \nonumber \\
				& = \sum_{j=1}^K \left(\bm{I}_K \otimes \b1_m \right)_{ij} \left( \sum_{1 \leq i_2,\dots,i_d \leq n}  \left(\bm{I}_K \otimes \b1_m \right)_{i_2j} \times \dots \times \left(\bm{I}_K \otimes \b1_m \right)_{i_dj}  \times \bm{H}_{i_2k} \times \cdots \times \bm{H}_{i_dk}  \right) \nonumber \\
				& = \sum_{j=1}^K \left(\bm{I}_K \otimes \b1_m \right)_{ij} \left( \b1_n^\top \left( \left(\bm{I}_K \otimes \b1_m \right)_j \circ \BH_k   \right) \right)^{d-1} \nonumber \\
				& = \sum_{j=1}^K (\bm{I}_K \otimes \b1_m)_{ij} D_{jk}. \label{eq: D-def}
			\end{align}
			Here, $``\circ"$ represents element-wise multiplication,  $\BH_{j,k} \in \Br^{m}$ denotes the subvector of $\BH_k$ from $H_{i_1 k}$ to $H_{i_2 k}$ with $i_1 = (j-1) \cdot m+1$ and $i_2 = j \cdot m$, and $\bm{D} \in \Br^{K \times K}$ is the matrix given by $D_{jk} = (\b1_m^\top \BH_{j,k} )^{d-1}$ for $j\in [K], k \in [K]$.
			One can verify that $\bm{D} = m^{d-1}\BQ \bm{Z}$ according to \eqref{eq: Zexpression}.
			Combining \eqref{linearoperator-zero}, \eqref{eq: D-def} as well as \eqref{eq: eadecom} yields
			\begin{align*}
				(\BE[\ACal] )\left[\bm{H}^{\otimes (d-1)}\right] - (\BE[\ACal] )\left[ (\bm{H}^*\BQ)^{\otimes (d-1)}\right]
				& = (p-q)\cdot \left( \left(\left(\BH^*\right)^{\otimes d} \right) \left[\bm{H} ^{\otimes (d-1)}\right] -  \left(\left(\BH^*\right)^{\otimes d} \right)\left[(\bm{H}^*\bm{Q} )^{\otimes (d-1)}\right] \right)  \\
				& = (p-q) \cdot (\bm{I}_K \otimes \b1_m)(m^{d-1}\BQ\bm{Z}- m^{d-1}\BQ).
			\end{align*}
			This implies that
			\begin{align}
				\left\| (\BE[\ACal] )\left[\bm{H}^{\otimes (d-1)}\right] - (\BE[\ACal] )\left[ (\bm{H}^*\BQ)^{\otimes (d-1)}\right] \right\|_F
				\leq \sqrt{m}(p-q)m^{d-1}\| \bm{I}_K - \bm{Z} \|_F. \label{eq: fnorm-expA}
			\end{align}
			Next, we provide an upper bound for the term $\| \bm{I}_K - \bm{Z} \|_F$.
			Note that
			\begin{align} \label{eq: Ik-Zeq}
				\| \bm{I}_K - \bm{Z}\|_F^2 = \sum_{k=1}^K (1-Z_{kk})^2 + \sum_{k\neq \ell} Z_{k\ell}^2.
			\end{align}
			Since $Z_{k\ell} \in [0,1]$ for $k,\ell \in [K]$, it follows that
			\begin{align}
				\| \bm{I}_K - \bm{Z}\|_F & \leq \sum_{k=1}^K |1-Z_{kk}| + \sum_{k\neq \ell}|Z_{k\ell}| = \sum_{k=1}^K (1-Z_{kk}) + \sum_{k\neq \ell} Z_{k\ell} \nonumber \\
				&\leq \sum_{k=1}^K (1-Z_{kk}) + \sum_{k\neq \ell} Z_{k\ell}^{\frac{1}{d-1}}. \label{eq: Ik-zupperbound0}
			\end{align}
			By the definition of $\bm{Z}$ in \eqref{eq: Zexpression}, we have
   \begin{equation}\label{eq:sumz1}
       \sum_{\ell=1}^K Z_{k \ell}^{\frac{1}{d-1}} = \frac{1}{m}\sum_{\ell=1}^K \langle \bm{H}_\ell, (\bm{H}^*\bm{Q})_k \rangle = 1
   \end{equation}
   for $k \in [K]$. Then, we can further estimate \eqref{eq: Ik-zupperbound0} by
       \begin{align}
				\| \bm{I}_K - \bm{Z}\|_F \leq& \sum_{k=1}^K (1-Z_{kk}) + \sum_{k=1}^K \left(1-Z_{kk}^{\frac{1}{d-1}}\right)
				\leq 2 \sum_{k=1}^K (1-Z_{kk}). \label{eq: Ik-zupperbound}
			\end{align}
       Here, again, we have used the fact that $Z_{kk} \in [0,1]$ for $k\in [K]$.
			The orthogonal decomposition in \eqref{eq: Horthdecom} gives
			\begin{align} \label{eq: orth-decom-norm}
				\sum_{j=1}^K \langle \GCal_j, \GCal_j \rangle = m^{d-1}K - m^{d-1} \sum_{k=1}^K \sum_{j=1}^K Z_{kj}^2.
			\end{align}
			In addition, using the fact that
			\begin{equation*}
				m^{d-1} - s^{d-1} = (m-s)\left(m^{d-2} + m^{d-3}s + \dots + ms^{d-3} + s^{d-2}\right) \text{ for } 0 \leq s \leq m,
			\end{equation*} one can verify the following inequalities:
			\begin{align} \label{eq: kroneck-ineq}
				m^{d-2} \| \BH - \BH^* \BQ\|_F^2 \leq \sum_{k=1}^K \left\| (\BH_k)^{\otimes (d-1)} - ((\BH^* \BQ)_k)^{\otimes (d-1)} \right\|_F^2 \leq (d-1)m^{d-2}\| \BH - \BH^* \BQ\|_F^2.
			\end{align}
			According to $\| \BH - \BH^* \BQ\|_F \leq \varepsilon \sqrt{n}$, we have
			\begin{align*}
				\sum_{k=1}^K \| (\BH_k)^{\otimes (d-1)} - ((\BH^* \BQ)_k)^{\otimes (d-1)}\|_F^2  & = m^{d-1} \left(\sum_{k=1}^K (1- Z_{kk})^2 + \sum_{k\neq j} Z_{kj}^2 \right) + \sum_{j=1}^K \langle \GCal_j, \GCal_j \rangle \\
				& = m^{d-1} \|\bm{I}_K - \bm{Z} \|_F^2 + \sum_{j=1}^K \langle \GCal_j, \GCal_j \rangle \leq (d-1)m^{d-2}\varepsilon^2 n,
			\end{align*}
			where the first equality is due to \eqref{eq: Horthdecom} and the inequality comes from \eqref{eq: kroneck-ineq}.
			This, together with \eqref{eq: Ik-Zeq} and \eqref{eq: orth-decom-norm}, implies that
			\begin{align}
				\sum_{k=1}^K Z_{kk} \geq \left(1-\frac{(d-1)\varepsilon^2}{2}\right)K.
			\end{align}
			Then, for any $\ell \in [K]$, we have
			\begin{align}\label{eq: zkk1over2}
				Z_{\ell\ell} \geq \left(1-\frac{(d-1)\varepsilon^2}{2}\right)K - \sum_{k\neq \ell}Z_{kk}  \geq 1- \frac{K}{2}(d-1)\varepsilon^2 \geq \frac{1}{2},
			\end{align}
			where the second inequality is due to $Z_{kk}\leq 1$ for all $k \in [K]$ and the last inequality is from $\varepsilon \in (0,1/\sqrt{K(d-1)})$.
			According to \eqref{eq: orth-decom-norm}, we know that
			\begin{align}
				\frac{1}{m^{d-1}}\sum_{j=1}^K \langle \GCal_j, \GCal_j \rangle &=  K - \sum_{k=1}^K \sum_{\ell=1}^K Z_{k\ell}^2 \geq K - \sum_{k=1}^K Z_{kk}^2 - \sum_{k\neq \ell}Z_{k\ell}^{\frac{1}{d-1}} \nonumber \\
				& =  K - \sum_{k=1}^K Z_{kk}^2  - \sum_{k=1}^K \left(1-Z_{kk}^{\frac{1}{d-1}} \right) = \sum_{k=1}^K Z_{kk}^{\frac{1}{d-1}}\left(1-Z_{kk}^{2-\frac{1}{d-1}}\right) \nonumber \\
				& \geq \sum_{k=1}^K\left( \frac{1}{2}\right)^\frac{1}{d-1} \left(1-Z_{kk} \right),
			\end{align}
			where the first inequality is due to $Z_{k\ell} \in [0,1]$ for $k\in[K], \ell \in [K]$, the second equality comes from \eqref{eq:sumz1}, and the last inequality uses \eqref{eq: zkk1over2} and $2-\frac{1}{d-1}\geq 1$ when $d \geq 2$. This, together with \eqref{eq: Ik-zupperbound} and \eqref{eq: kroneck-ineq}, gives
			\begin{align*}
				\| \bm{I}_K - \bm{Z}\|_F & \leq	2^{\frac{1}{d-1}}\frac{2}{m^{d-1}}\sum_{j=1}^K \langle \GCal_j, \GCal_j \rangle \leq \frac{4}{m^{d-1}} 	\sum_{k=1}^K \left\| (\BH_k)^{\otimes (d-1)} - ((\BH^* \BQ)_k)^{\otimes (d-1)} \right \|_F^2\\
				& \leq \frac{4}{m^{d-1}} (d-1)m^{d-2}\| \BH - \BH^* \BQ\|_F^2 \leq \frac{4(d-1)\varepsilon \sqrt{n}}{m} \| \BH - \BH^*\BQ\|_F.
			\end{align*}
			Next, we consider the term $\left \| \Delta \left[\bm{H}^{\otimes (d-1)}\right] - \Delta \left[(\bm{H}^*\bm{Q})^{\otimes (d-1)}\right] \right\|_F$ and provide an upper bound for it. Note that
			\begin{align}\label{eq: boundnoise}
				\left\| \Delta \left[\bm{H}_k^{\otimes (d-1)}\right] - \Delta \left[(\bm{H}^*\bm{Q})_k^{\otimes (d-1)}\right] \right\|_2 = \left\| \mathcal{M}(\Delta) \left(\bm{H}_k^{\odot (d-1)} - (\bm{H}^*\bm{Q})_k^{\odot (d-1)} \right) \right\|_2.
			\end{align}
			From Lemma \ref{prop: concentration}, we know that
			\begin{equation}
				\left\| \mathcal{M}(\Delta) \left(\bm{H}_k^{\odot (d-1)} - (\bm{H}^*\bm{Q})_k^{\odot (d-1)} \right) \right\|_2 \leq C \sqrt{\log n} \| \BH_k - (\BH^*\BQ)_k \|_2,
			\end{equation}
			which implies that
			\begin{align}
				\left \| \Delta \left[\bm{H}^{\otimes (d-1)}\right] - \Delta \left[(\bm{H}^*\bm{Q})^{\otimes (d-1)}\right] \right\|_F \leq C \sqrt{\log n} \|\bm{H} - \bm{H}^*\bm{Q}\|_F.
			\end{align}
			The desired result \eqref{eq: multi-linear} is then established, because 
			\begin{align*}
				& \quad\, \left\| \ACal\left[\bm{H}^{\otimes (d-1)}\right] - \ACal\left[ (\bm{H}^*\BQ)^{\otimes (d-1)}\right] \right\|_F \\
				& \leq  \left\| (\BE[\ACal] )\left[\bm{H}^{\otimes (d-1)}\right] - (\BE[\ACal] )\left[ (\bm{H}^*\BQ)^{\otimes (d-1)}\right] \right\|_F +  \left\| \Delta \left[\bm{H}^{\otimes (d-1)}\right] - \Delta \left[(\bm{H}^*\bm{Q})^{\otimes (d-1)}\right] \right\|_F \\
				& = \left(\frac{4(d-1)m^{d-2}\varepsilon n}{\sqrt{K}} (p-q) + C\sqrt{\log n} \right) \|\bm{H} - \bm{H}^*\bm{Q}\|_F.
			\end{align*}
   This completes the proof.
		\end{proof}
		\section{Proof of Lemma \ref{prop: concentration}}
		\begin{proof}
			By replacing $\by$ with $\frac{\by}{\| \by\|_2}$ in \eqref{eq:concentration}, it suffices to prove that for any $r > 0$ and any vector $\by\in \Br^{n^{d-1}}$ with $\ell \cdot n^{d-2}$ nonzero elements, in which half of them take values of $1/\sqrt{\ell \cdot n^{d-2}}$ and others take $-1/\sqrt{\ell \cdot n^{d-2}}$, there exists a constant $C>0$ such that 
			\begin{align*}
				\| (\BA - \BP) \by \|_2 \leq C\sqrt{\xi/n^{d-2}}.
			\end{align*}
			The fact that
			\begin{equation} \label{eq: sup-infty}
				\| (\BA - \BP) \by \|_2 = \sup_{\bx \in \Br^n, \| \bx\|_2 \leq 1} \bx^\top (\BA - \BP) \by
			\end{equation}
			motivates us to utilize the Kahn-Szemer{\'e}di argument to provide an upper bound for \eqref{eq: sup-infty}.  The idea of the proof is to discretize the set $\{\bx \in \Br^n \colon \| \bx\|_2 \leq 1\}$ into a finite set of grid points and estimate the supremum of $\bm{x}^T(\bm{A} - \bm{P})\bm{y}$ by dividing the pairs of vectors $(\bm{x}, \bm{y})$ into two parts: (i) the small entries of $\bm{x}$ and $\bm{y}$, which we call the \emph{light part}; and (ii) the larger entries of $\bm{x}$ and $\bm{y}$, which we call the \emph{heavy part}.
			
			\paragraph{Discretization}
			Fix $\delta \in (0,1)$, for example $\delta = \frac{1}{2}$, and define the sets $S_t \coloneqq \{\bx \in \Br^n: \| \bx\|_2 \leq t\}$ and
			\[T \coloneqq \left\{\bx=\left(x_1, \ldots, x_n\right)^\top \in S_1: \sqrt{n} x_i / \delta \in \mathbb{Z}, \forall i\in [n] \right\}.
			\]
			Following \citet[ Supplementary material: Lemma 2.1]{lei2015consistency}, we have the following inequality:
			\begin{equation}\label{eq:discretization}
				\| (\BA - \BP) \by \|_2 = \sup_{\bx \in \Br^n, \| \bx\|_2 \leq 1} \bx^\top (\BA - \BP) \by \leq (1-\delta)^{-1} \max_{\bx \in T} \left|\bx^\top (\BA - \BP) \by \right|.
			\end{equation}
			Note that, for any vector $\bx \in T$,
			\[
			\bx^\top(\BA - \BP) \by=\sum_{i\in [n],  j\in [n^{d-1}]} x_i y_j\left(a_{i j}-p_{i j}\right) .
			\]
			Consider the light pairs
			\[
			\mathcal{L}=\mathcal{L}(\bx, \by) \coloneqq \left\{(i, j):\left|x_i y_j\right| \leq \sqrt{\frac{\xi}{n^{d}}} \right\}
			\]
			and heavy pairs
			\[
			\overline{\mathcal{L}}=\overline{\mathcal{L}}(\bx, \by) \coloneqq \left\{(i, j):\left|x_i y_j\right|>\sqrt{\frac{\xi}{n^{d}}}\right\}.
			\]
			Then, it follows from \eqref{eq:discretization} that
			\begin{align*}
				\| (\BA - \BP) \by \|_2 & \leq (1-\delta)^{-1} \max_{\bx \in T} \left|\sum_{i\in [n],  j\in [n^{d-1}]} x_i y_j\left(a_{i j}-p_{i j}\right) \right| \\
				& \leq (1-\delta)^{-1} \left( \max_{\bx \in T} \left|\sum_{(i,j)\in \LCal}x_i y_j\left(a_{i j}-p_{i j}\right) \right| + \max_{\bx \in T} \left|\sum_{(i,j)\in \overline{\mathcal{L}}}x_i y_j\left(a_{i j}-p_{i j}\right) \right| \right).
			\end{align*}
			\paragraph{Bounding the light part} Denote $\BW \coloneqq \BA - \BP$ and $u_{i j} \coloneqq x_i y_j \mathbbm{1}_{\left(\left|x_i y_j\right| \leq \sqrt{\xi/n^d} \right)}$ for $i\in[n], j\in[n^{d-1}]$. The light part of $\bx^\top \BW \by$ is given by
			\begin{equation*}
				\sum_{i\in[n],j\in[n^{d-1}]} w_{i j} u_{i j}.
			\end{equation*}
			Since $\left|u_{i j}\right| \leq \sqrt{\xi/n^d}$, the term $w_{ij}u_{ij}$ is of mean zero and bounded in absolute value by $\sqrt{\xi/n^d}$. By Bernstein's inequality \citep[Theorem 2.8.4]{vershynin2018high}, we have
			\begin{align*}
				\operatorname{Pr}\left[\left|\sum_{i,j} w_{i j} u_{i j}\right| \geq c \sqrt{\frac{\xi}{n^{d-2}}}\right] & \leq 2 \exp \left(\frac{-\frac{1}{2} c^2 \xi/n^{d-2}}{\sum_{i,j} p_{i j}\left(1-p_{i j}\right) u_{i j}^2+ \frac{1}{3}\sqrt{\frac{\xi}{n^d}} c \sqrt{\frac{\xi}{n^{d-2}}}}\right) \\
				& \leq 2 \exp \left(\frac{-\frac{1}{2} c^2 \xi/n^{d-2}}{p_{\max } \sum_{i,j} u_{i j}^2+ \frac{c\xi}{3n^{d-1}}}\right) \\
				& \leq 2 \exp \left(\frac{-3 c^2 \xi/n^{d-2}}{6\xi/n^{d-1}+ \frac{2c\xi}{n^{d-1}}}\right)
				\leq 2 \exp \left(\frac{-3c^2}{6+2c}n \right).
			\end{align*}
			The third inequality follows from the facts that $\xi \geq n^{d-1} p_{\max }$ and
			\[
			\sum_{i,j} u_{i j}^2 \leq \sum_{ i, j} x_i^2 y_j^2=\|\bx\|_2^2\|\by\|_2^2 \leq 1 .
			\]
			Applying the union bound and the volume bound $|T| \leq e^{n \log (7 / \delta)}$ \citep{lei2015consistency}, we obtain
			\begin{align}
				\operatorname{Pr}\left[ \max_{\bx \in T} \left|\sum_{ (i,j)\in {\mathcal{L}(\bx, \by) }} x_i y_j w_{ij} \right| \geq c \sqrt{\frac{\xi}{n^{d-2}}}\right] \leq 2 \exp \left( -\left(\frac{3c^2}{6+2c} - \log\left(\frac{7}{\delta}\right) \right)n \right).
			\end{align}
			\paragraph{Bounding the heavy part} The more challenging part of the proof is to show that the heavy part
			\[
			\max_{\bx \in T} \left|\sum_{ (i,j)\in {\overline{\mathcal{L}}(\bx, \by) }} x_i y_j w_{ij} \right|
			\]
			is upper bounded by $c \sqrt{\frac{\xi}{n^{d-2}}}$ with high probability for some universal constant $c$. Observe that the expectation of $\sum_{ (i,j)\in {\overline{\mathcal{L}}(\bx, \by) }} x_i y_j a_{ij}$ can be well controlled:
			\begin{align*}
				\left|\sum_{ (i,j)\in {\overline{\mathcal{L}}(\bx, \by) }} x_i y_j p_{ij} \right| & = \left|\sum_{ (i,j)\in {\overline{\mathcal{L}}(\bx, \by) }} \frac{x_i^2 y_j^2}{x_iy_j} p_{ij} \right| \leq \sum_{ (i,j)\in {\overline{\mathcal{L}}(\bx, \by) }} \frac{x_i^2 y_j^2}{|x_iy_j|} p_{ij}\\
				& \leq p_{\max} \sqrt{\frac{n^{d}}{\xi}}\sum_{ (i,j)\in {\overline{\mathcal{L}}(\bx, \by) }}x_i^2 y_j^2 \leq p_{\max} \sqrt{\frac{n^{d}}{\xi}} = \sqrt{\frac{\xi}{n^{d-2}}},
			\end{align*}
   where the second inequality comes from the definition of heavy pairs.
			Then, it suffices to show that 
			\begin{align} \label{eq: heavypartgoal-a}
				\left|\sum_{ (i,j)\in {\overline{\mathcal{L}}(\bx, \by) }} x_i y_j a_{ij} \right| = \OCal\left(\sqrt{\frac{\xi}{n^{d-2}}}\right)
			\end{align}
			with high probability.
			
			We will focus on the heavy pairs $(i, j)$ such that $x_i>0, y_j>0$ and denote
			$$
			\overline{\mathcal{L}}_1 \coloneqq \left\{(i, j) \in \overline{\mathcal{L}}: x_i>0, y_j>0\right\}.
			$$
			The other three cases are similar. Notice that if $y_j \neq 0$, given the assumption that $\ell \leq n$, we have $|y_j| = \frac{1}{\sqrt{n^{d-1}}}\sqrt{ \frac{n}{\ell}} \geq \frac{\delta}{\sqrt{n^{d-1}}}$. In what follows, we use the following notation:
			\begin{itemize}
				\item $I_1 \coloneqq \left\{i: \frac{\delta}{\sqrt{n}} \leq x_i \leq \frac{2 \delta}{\sqrt{n}}\right\}, I_s \coloneqq \left\{i: \frac{\delta}{\sqrt{n}} 2^{s-1}<x_i \leq \frac{\delta}{\sqrt{n}} 2^s\right\}$ for $s=$ $2,3, \ldots,\left\lceil\log _2 \frac{\sqrt{n}}{\delta}\right\rceil$.
    \item $J_1 \coloneqq \left\{j: \frac{\delta}{\sqrt{n^{d-1}}} \leq y_j \leq \frac{2 \delta}{\sqrt{n^{d-1}}}\right\}, J_t \coloneqq \left\{j: \frac{\delta}{\sqrt{n^{d-1}}} 2^{t-1}<y_j \leq \frac{\delta}{\sqrt{n^{d-1}}} 2^t\right\}$ for $t=$ $2,3, \ldots,\left\lceil\log _2 \frac{\sqrt{n^{d-1}}}{\delta}\right\rceil$.
				\item $e(I, J) \coloneqq \sum_{i \in I, j \in J} a_{i j}$.
				\item $\mu(I, J) \coloneqq \mathbb{E}[e(I, J)],~\bar{\mu}(I, J) \coloneqq p_{\max }|I||J|$. For simplicity, we will use $\mu$ and $\bar{\mu}$ when we do not need to specify their dependence on $I$ and $J$.
				\item $\lambda_{s t} \coloneqq e\left(I_s, J_t\right) / \bar{\mu}_{s t}$, where $\bar{\mu}_{s t} \coloneqq \bar{\mu}\left(I_s, J_t\right)$.
				\item $\alpha_s \coloneqq \left|I_s\right| 2^{2 s} / n,~ \beta_t \coloneqq \left|J_t\right| 2^{2 t} / n^{d-1},~\sigma_{s t} \coloneqq \lambda_{s t} \sqrt{\xi} 2^{-(s+t)}$.
			\end{itemize}
			The following two lemmas are important to the rest of the proof.
			\begin{lemma}[Bounded degree]\label{lemma: rwosumconcentration}
				For $c>0$, there exists a constant $c_1=$ $c_1(c)$ such that with probability at least $1-n^{-c}$,
				\begin{equation} \label{eq: sumnk-1bou}
					\sum_{j\in[n^{d-1}]}a_{ij} \leq c_1 \xi~\text{ for all }~i \in [n].
				\end{equation}
			\end{lemma}
			\begin{proof}
				The result follows directly by applying Bernstein's inequality and the union bound.
			\end{proof}
			
			\begin{lemma}[Bounded discrepancy]\label{lem: boudis}
				For any $c>0$, there exist constants $c_2=c_2(c) >1$ and $c_3=c_3(c)>1$ such that with probability at least $1-2 n^{-c}$, for a fixed index set $J\subseteq [n^{d-1}]$ and any index set $I\subseteq [n]$ with $|I| \leq|J|/n^{d-2}$, at least one of the following hold:\\
				1) $\frac{e(I, J)}{\bar{\mu}(I, J)} \leq e c_2$;\\
				2) $e(I, J) \log \frac{e(I, J)}{\bar{\mu}(I, J)} \leq c_3\frac{|J|}{n^{d-2}} \log \frac{n^{d-1}}{|J|}$.
			\end{lemma}
			\begin{proof}
				Suppose that the event in \eqref{eq: sumnk-1bou} holds. If $|J|/n^{d-2} \geq n / e$, then \eqref{eq: sumnk-1bou} implies that $\frac{e(I, J)}{\xi|I||J| / n^{d-1}} \leq \frac{|I| c_1 \xi}{\xi|I| / e} \leq c_1 e$. If $|J|/n^{d-2}<n / e$, then by \citet[Lemma 4.5]{zhou2021sparse}, we have for any $\tau >1$,
				\begin{align}
					\operatorname{Pr}[e(I, J) \geq \tau \bar{\mu}(I, J)] & \leq \operatorname{Pr}\left[\sum_{i\in I, j\in J}\left(a_{i j}-p_{i j}\right) \geq \tau \bar{\mu}(I, J)-\sum_{i\in I, j\in J} p_{i j}\right] \nonumber \\
					& \leq \operatorname{Pr}\left[\sum_{i\in I, j\in J} w_{i j} \geq(\tau-1) \bar{\mu}(I, J)\right]  \nonumber \\
					& \leq \exp \left((\tau-1) \bar{\mu}-\tau \bar{\mu} \log \tau \right) \leq \exp \left[-\frac{1}{2}(\tau \log \tau) \bar{\mu}\right], \label{eq: taulogtau}
				\end{align}
				where the last inequality holds when $\tau \geq 8$.
				For any given $c_3>0$, let $t(I, J)$ denote the unique value of $t$ satisfying $t \log t=\frac{c_3|J|/n^{d-2}}{\bar{\mu}(I, J)} \log \frac{n^{d-1}}{\mid J \mid}$. Let $\tau(I, J)\coloneqq\max \{8, t(I, J)\}$. Then, by \eqref{eq: taulogtau}, we have
				\begin{align*}
						\operatorname{Pr}[e(I, J) \geq \tau(I, J) \bar{\mu}(I, J)] \leq \exp \left[-\frac{1}{2} \bar{\mu}(I, J) \tau(I, J) \log \tau(I, J)\right]
						\leq \exp \left[-\frac{1}{2} c_3\frac{|J|}{n^{d-2}} \log \frac{n^{d-1}}{|J|}\right].
				\end{align*}
				Let $\Omega \coloneqq \left\{(I,J): |I| \leq g =  \frac{|J|}{n^{d-2}} \leq \frac{n}{e} \right\}$. We bound
				\begin{align*}
					\operatorname{Pr}\left[\exists(I, J)\in \Omega, e(I, J) \geq \tau(I, J) \bar{\mu}(I, J)\right] 
					\leq & \sum_{(I,J)\in \Omega} \exp \left[-\frac{1}{2} c_3\frac{|J|}{n^{d-2}} \log \frac{n^{d-1}}{|J|}\right] \\
					\leq & \sum_{h: 1 \leq h \leq g \leq n / e} \,\sum_{I:|I|=h} \exp \left[-\frac{1}{2} c_3 g \log \frac{n}{g}\right] \\
					\leq &\sum_{h: 1 \leq h \leq g \leq n / e} \tbinom{n}{h} \exp \left[-\frac{1}{2} c_3 g \log \frac{n}{g}\right].
				\end{align*}
				Since $\tbinom{n}{h} \leq (\frac{ne}{h})^h$ for any integer $1\leq h \leq n$, the last line above is bounded by
				\begin{align*}
					\sum_{h: 1 \leq h \leq g \leq n / e}\left(\frac{n e}{h}\right)^h \exp \left[-\frac{1}{2} c_3 g \log \frac{n}{g}\right]
					=& \sum_{h: 1 \leq h \leq g \leq n / e} \exp \left[-\frac{1}{2} c_3 g \log \frac{n}{g}+h \log \frac{n}{h}+h\right] \\
					\leq & \sum_{h: 1 \leq h \leq g \leq n / e} \exp \left[-\frac{1}{2} c_3 g \log \frac{n}{g}+g \log \frac{n}{g}+ g\right] \\
					\leq & \sum_{h: 1 \leq h \leq g \leq n / e} \exp \left[-\frac{1}{2}\left(c_3-4\right) g \log \frac{n}{g}\right] \\
					\leq & \sum_{h: 1 \leq h \leq g \leq n / e} n^{-\frac{1}{2}\left(c_3-4\right)} \leq n^{-\frac{1}{2} \left(c_3-6\right)},
				\end{align*}
				where the inequalities repeatedly use the assumption that $h \leq g \leq n / e$ and the fact that $t \mapsto t \log \frac{n}{t}$ is increasing on $[1, n / e]$.
				
				Hence, with probability at least $1-n^{-\frac{1}{2}\left(c_3-6\right)}$, we know $e(I, J) \leq$ $\tau(I, J) \bar{\mu}(I, J)$ for all $|I| \leq|J|/n^{d-2} \leq n /e$. We further divide the set of pairs $(I, J)$ satisfying $|I| \leq|J|/n^{d-2} \leq n / e$ into two groups by the value of $\tau(I, J)$. For the pairs satisfying $\tau(I,J)=8$, we get
				\begin{equation*}
					e(I, J) \leq \tau(I, J) \bar{\mu}(I, J)=8 \bar{\mu}(I, J).
				\end{equation*}
				For all the other pairs, we have $\tau(I, J)=t(I, J)>8$ and $\frac{e(I, J)}{\bar{\mu}(I, J)} \leq$ $t(I, J)$. It follows that
				$$
				\frac{e(I, J)}{\bar{\mu}(I, J)} \log \frac{e(I, J)}{\bar{\mu}(I, J)} \leq t(I, J) \log t(I, J)=\frac{c_3|J|/n^{d-2}}{\bar{\mu}(I, J)} \log \frac{n^{d-1}}{|J|},
				$$
				which implies that
				$$
				e(I, J) \log \frac{e(I, J)}{\bar{\mu}(I, J)} \leq c_3\frac{|J|}{n^{d-2}}  \log \frac{n^{d-1}}{|J|} .
				$$
				The desired result follows by letting $c_2=\max \left\{c_1, 8\right\}$ and $c_3=2 c+6$.
			\end{proof}
			
			Now, we write the left-hand side of \eqref{eq: heavypartgoal-a} as
			\begin{align}
				\sum_{(i, j) \in \overline{\mathcal{L}}_1} x_i y_j a_{i j} & \leq \sum_{(s, t):\, 2^{s+t} \geq \sqrt{\xi}} e\left(I_s, J_t\right) \frac{2^s \delta}{\sqrt{n}} \frac{2^t \delta}{\sqrt{n^{d-1}}} = \delta^2 \sqrt{\frac{\xi}{n^{d-2}}} \sum_{(s, t): \,2^{s+t} \geq \sqrt{\xi}} \alpha_s \beta_t \sigma_{s t}. \label{eq: heavysum}
			\end{align}
			We estimate this sum by splitting the pairs of $(s, t)$ into twelve different categories. Denote by $J_{t'}$ the nonempty set in $\left\{J_t:t=1,\dots,\left\lceil\log _2 \frac{\sqrt{n^{d-1}}}{\delta}\right\rceil\right\}$. Let
			\[
			\mathcal{C}\coloneqq \left\{s: 2^{s+t'} \geq \sqrt{\xi},\left|I_s\right| \leq\left|J_{t'}\right|/n^{d-2}\right\}
			\] and define the following sets:
			\begin{itemize}
				\item $\mathcal{C}_1 \coloneqq \left\{s \in \mathcal{C}: \sigma_{s t'} \leq 1 \right\}$,
				\item $\mathcal{C}_2\coloneqq \left\{ s \in \mathcal{C} \backslash \mathcal{C}_1: \lambda_{s t'} \leq e c_2\right\}$,
				\item $\mathcal{C}_3\coloneqq \left\{s\in \mathcal{C} \backslash\left(\mathcal{C}_1 \cup \mathcal{C}_2\right): 2^s \geq \sqrt{\xi} 2^{t'}\right\}$,
				\item $\mathcal{C}_4\coloneqq\left\{s \in \mathcal{C} \backslash\left(\mathcal{C}_1 \cup \mathcal{C}_2 \cup \mathcal{C}_3\right): \log \lambda_{st'} > \frac{1}{4} \left(2t' \log 2 + \log ( \beta_{t'}^{-1} ) \right)\right\}$,
				\item $\mathcal{C}_5\coloneqq \left\{s \in \mathcal{C} \backslash\left(\mathcal{C}_1 \cup \mathcal{C}_2 \cup \mathcal{C}_3 \cup \mathcal{C}_4\right): 2 t' \log 2 \geq \log ( \beta_{t'}^{-1} ) \right\}$,
				\item $\mathcal{C}_6\coloneqq \left\{s \in \mathcal{C} \backslash\left(\mathcal{C}_1 \cup \mathcal{C}_2 \cup \mathcal{C}_3 \cup \mathcal{C}_4 \cup \mathcal{C}_5\right)\right\}$.
			\end{itemize}
			The other six categories can be defined using a similar partition of
			$$
			\mathcal{C}^{\prime}\coloneqq \left\{s: 2^{s+t'} \geq \sqrt{\xi},\left|I_s\right|>\left|J_{t'}\right|/n^{d-2}\right\}
			$$
			and can be analyzed similarly. We now analyze each of the six cases separately. To that end, we will repeatedly make use of the following estimates:
			\[
			\sum_s \alpha_s \leq \sum_i\left|2 x_i / \delta\right|^2 \leq 4 \delta^{-2}, \beta_{t'} \leq l\cdot n^{d-2} 2^{2 t'} / n^{d-1},~\text{and}~\beta_t = 0 \text{ for } t \neq t'.
			\]
			In addition, we have $\beta_{t'} = \sum_t \beta_t \leq 4 \delta^{-2}$.
			
			\textbf{Indices in} $\mathbf{\CCal_1}$: Since $\sigma_{st'}\le 1$, we have
			\begin{align}
				\sum_{s} \alpha_s \beta_{t'} \sigma_{s t'} \mathbbm{1}_{\left(s \in \mathcal{C}_1\right)} \leq 16 \delta^{-4}.
			\end{align}
			
			\textbf{Indices in} $\mathbf{\CCal_2}$: By definition, it holds that $\sigma_{s t'}=\lambda_{s t'} \sqrt{\xi} 2^{-(s+t')} \leq \lambda_{st'} \leq ec_2$. Hence,
			\begin{align}
				\sum_{s} \alpha_s \beta_{t'} \sigma_{s t'} \mathbbm{1}_{\left(s \in \mathcal{C}_2\right)} \leq ec_2 16 \delta^{-4}.
			\end{align}
			
			\textbf{Indices in} $\mathbf{\CCal_3}$: We know from Lemma \ref{lemma: rwosumconcentration} that
			$e(I_s,J_{t'}) \leq c_1 |I_s|\xi$ holds with probability at least $1-n^{-c}$. It follows that $\lambda_{st'} \leq c_1 n^{d-1}/|J_{t'}|$ and
			\begin{align*}
				\sum_{s} \alpha_s \beta_{t'} \sigma_{s t'} \mathbbm{1}_{\left(s \in \mathcal{C}_3\right)} \leq \sum_{s} \alpha_s \frac{|J_{t'}|2^{2t'}}{n^{d-1}} \frac{c_1 n^{d-1}}{|J_{t'}|} \sqrt{\xi} 2^{-(s+t')} \mathbbm{1}_{\left(s \in \mathcal{C}_3\right)}
				\leq c_1 \sum_s \alpha_s \frac{\sqrt{\xi}}{2^{s-t'}}\mathbbm{1}_{\left(s \in \mathcal{C}_3\right)}  \leq 4c_1 \delta^{-2},
			\end{align*}
   where the last inequality comes from $2^{s-t'} \ge \sqrt{\xi}$.
			
			To cope with $\CCal_4,\CCal_5,$ and $\CCal_6$, we rely on the second case in Lemma \ref{lem: boudis}, which is equivalent to
			\begin{align}
				& \lambda_{s t'}\left|I_s \| J_{t'}\right| \frac{\xi}{n^{d-1}} \log \lambda_{s t'} \leq c_3\frac{\left|J_{t'}\right|}{  n^{d-2}} \log \frac{n^{d-1}}{\left|J_{t'}\right|} \nonumber \\
				\iff \; & \sigma_{s t'} \alpha_s \log \lambda_{s t'} \leq c_3 \frac{2^{s-t'}}{\sqrt{\xi}}\log \frac{n^{d-1}}{\left|J_{t'}\right|} \nonumber \\
				\iff \; & \sigma_{s t'} \alpha_s \log \lambda_{s t'} \leq c_3 \frac{2^{s-t'}}{\sqrt{\xi}} { \left(2t' \log 2 + \log ( \beta_{t'}^{-1} ) \right)}. \label{eq: c4c5c6}
			\end{align}
			\textbf{Indices in} $\mathbf{\CCal_4}$:
			The inequality $\log \lambda_{st'} > \frac{1}{4} \left( 2t' \log 2 + \log ( \beta_{t'}^{-1} ) \right)$ and \eqref{eq: c4c5c6} imply that $\sigma_{st'} \alpha_s\leq 4c_3 2^{s-t'}/\sqrt{\xi} $. Then, we have
			\begin{align*}
				\sum_{s} \alpha_s \beta_{t'} \sigma_{s t'} \mathbbm{1}_{\left(s \in \mathcal{C}_4\right)}  =  \beta_{t'} \sum_{s} \alpha_s \sigma_{st'} \mathbbm{1}_{\left(s \in \mathcal{C}_4\right)} \leq 8c_3 \beta_{t'} \leq 32 c_3 \delta^{-2},
			\end{align*}
			where the first inequality is from $s \notin \CCal_3$.
			
			\textbf{Indices in} $\mathbf{\CCal_5}$:
			In this case, we have $2 t' \log 2 \geq \log (\beta_{t'}^{-1})$. Also, since $s \notin$ $\mathcal{C}_4$, we have $\log \lambda_{s t'} \leq \frac{1}{4}\left(2 t' \log 2+\log (\beta_{t'}^{-1}) \right) \leq t' \log 2$. Thus, $\lambda_{s t'} \leq 2^{t'}$. Besides, since $s \notin \mathcal{C}_1$, we have $1 \leq \sigma_{s t'}=\lambda_{s t'} \sqrt{\xi} 2^{-(s+t')} \leq \sqrt{\xi} 2^{-s}$. It follows that $2^s \leq \sqrt{\xi}$.
			
			Since $s\notin \mathcal{C}_2$, we have $\log \lambda_{s t'} \geq 1$. Together with $2 t' \log 2 \geq \log \beta_{t'}^{-1}$, \eqref{eq: c4c5c6} implies that
			\[
			\sigma_{s t'} \alpha_s \leq c_3 \frac{2^{s-t'}}{\sqrt{\xi}} 4 t' \log 2.
			\]
			Hence,
			\begin{align*}
				\sum_{s} \alpha_s \beta_{t'} \sigma_{s t'} \mathbbm{1}_{\left(s \in \mathcal{C}_5\right)} &=\beta_{t'} \sum_s \alpha_s \sigma_{s t'} \mathbbm{1}_{\left(s \in \mathcal{C}_5\right)} \leq \beta_{t'} \sum_s c_3 \frac{2^{s-t'}}{\sqrt{\xi}} 4 t'(\log 2) \mathbbm{1}_{\left(s \in \mathcal{C}_5\right)} \\
				& \leq 4 c_3 (\log 2) \beta_{t'} t' 2^{-t'} \sum_s  \frac{2^s}{\sqrt{\xi}} \mathbbm{1}_{\left(s \in \mathcal{C}_5\right)} \leq 4 c_3 (\log 2) \beta_{t'} \leq 16 c_3 \delta^{-2}.
			\end{align*}
			\textbf{Indices in} $\mathbf{\CCal_6}$: In this case, we have $2 t' \log 2<\log (\beta_{t'}^{-1})$. Since $s \notin \mathcal{C}_2 \cup \CCal_4$, we have $\log \lambda_{s t'} \geq 1$ and $\log \lambda_{s t'} \leq \frac{1}{2} \log \beta_{t'}^{-1} \leq \log \beta_{t'}^{-1}$. Now, we compute
			\begin{align*}
				\sum_{s} \alpha_s \beta_{t'} \sigma_{s t'} \mathbbm{1}_{\left(s \in \mathcal{C}_6\right)} &=\sum_s \alpha_s  \beta_{t'} \lambda_{s t'} \sqrt{\xi} 2^{-(s+t')} \mathbbm{1}_{\left(s \in \mathcal{C}_6\right)} \\
				& \leq \sum_s \alpha_s \sqrt{\xi} 2^{-(s+t')} \mathbbm{1}_{\left(s \in \mathcal{C}_6\right)} \leq \sum_s \alpha_s \leq 4 \delta^{-2}.
			\end{align*}
			The proof is complete.
		\end{proof}

\section{Proof of Remark \ref{remark: diagonalofA}}
\begin{proof}\label{proof: remark4.4}
            Let $\RCal \in S^d(\Br^n)$ be a tensor, whose entries are given by
	\begin{align*}
		\RCal_{i_1,i_2,\ldots,i_d} =
		\begin{cases}
            p,~&\text{if~there~exist}~i_j,i_k~\text{such~that}~i_j=i_k~\text{and}~
			i_1,\ldots,i_d~\text{belong~to~the~same~community};\\
			q,& \text{if~there~exist}~i_j,i_k~\text{such~that}~i_j=i_k~\text{and}~
			i_1,\ldots,i_d~\text{belong~to~different~communities};\\
			0,&\text{otherwise.}
		\end{cases}
	\end{align*}
	Then, the expectation of the adjacency tensor defined in Definition \ref{def: dHSBM} can be written as
	\begin{equation}
		\EE[\ACal] = q\cdot \left(\b1_n^{\otimes d}\right) + (p-q)\cdot \left(\BH^* \right)^{\otimes d} - \RCal.
	\end{equation}
	Following the argument in \eqref{eq: fnorm-expA}, it remains to show that the effect brought by the multilinear operator of the tensor $\RCal$ on the estimation error is negligible. Indeed, consider
	\begin{align}
		\left\| \RCal\left[\bm{H}^{\otimes (d-1)}\right] - \RCal\left[ (\bm{H}^*\BQ)^{\otimes (d-1)}\right] \right\|_F^2 \nonumber & = \left\| \MCal( \RCal) \left(\bm{H}^{\odot (d-1)} - (\bm{H}^*\bm{Q})^{\odot (d-1)} \right) \right\|_F^2\nonumber\\
		& \leq \left\| \MCal( \RCal) \right\|_F^2 \left\| \bm{H}^{\odot (d-1)} - (\bm{H}^*\bm{Q})^{\odot (d-1)} \right\|_F^2. \label{eq:diag-res}
	\end{align}
	Then, on the one hand, since the number of nonzero elements of $\RCal$ is at most $n\binom{d}{2}n^{d-2}$, we can estimate $\|\MCal(\RCal)\|_F^2$ by
	\begin{equation*}
		\|\MCal (\RCal) \|_F^2 \le p^2 \binom{d}{2}n^{d-1}.
	\end{equation*}
	On the other hand, upon applying \eqref{eq: kroneck-ineq}, we see that
	\begin{align*}
		\quad \left\| \bm{H}^{\odot (d-1)} - (\bm{H}^*\bm{Q})^{\odot (d-1)} \right\|_F^2 = \sum_{k=1}^K\left\| \bm{H}_k^{\otimes (d-1)} - (\bm{H}^*\bm{Q})_k^{\otimes (d-1)} \right\|_F^2 \le (d-1)m^{d-2}\|\bm{H} - \bm{H}^*\BQ\|_F^2.
	\end{align*}
	Combining the above two facts, \eqref{eq:diag-res} then implies that
	\begin{align*}
		\quad \left\| \RCal\left[\bm{H}^{\otimes (d-1)}\right] - \RCal\left[ (\bm{H}^*\BQ)^{\otimes (d-1)}\right] \right\|_F
		 \leq \OCal(pn^{d-3/2}) \cdot \|\bm{H} - \bm{H}^*\BQ\|_F,
	\end{align*}
	which is dominated by $\OCal((p-q)n^{d-1}) \cdot \|\bm{H} - \bm{H}^*\BQ\|_F $ (i.e., the first term on the right-hand side of \eqref{eq: multi-linear}).
\end{proof}

 \section{Proof of Lemma \ref{lemma: ground truth}} 
		\begin{proof}
			Since $\ACal$ is generated according to the $d$-HSBM, we know that for all $i \in \ICal_k$ with $\ell\neq k$,
			\begin{align}
				C_{ik} - C_{i\ell} & = \sum_{1 \leq i_2,\dots,i_d \leq n} \ACal_{i,i_2,\dots,i_d} H^*_{i_2k}\dots H^*_{i_dk} - \sum_{1 \leq i_2,\dots,i_d \leq n} \ACal_{i,i_2,\cdots,i_d} H^*_{i_2\ell}\cdots H^*_{i_d\ell} \\
				& \overset{\textnormal{d}}{=} (d-1)! \cdot \sum_{i=1}^{\binom{m-1}{d-1}} W_i - (d-1)! \cdot \sum_{i=1}^{\binom{m}{d-1}} Z_i, \label{eq: diffbinomial}
			\end{align}
			where $\{W_i\}$ are i.i.d. $\mathbf{Bern}(\alpha \log n / n^{d-1})$ and $\left\{Z_{i}\right\}$ are i.i.d. $\mathbf{Bern}(\beta \log n / n^{d-1})$ that are independent of $\{W_i\}$. By Lemma \ref{lemma: difbinom} and the fact that $n$ is sufficiently large, we have
			\begin{align}
				{\rm Pr}\left[C_{i k}-C_{i \ell} \geq \gamma \log n, \forall i \in \mathcal{I}_{k}, 1 \leq k \neq \ell \leq K\right] \geq 1-K n^{-c_{2}/2}
			\end{align}
			with 
			\begin{equation} \label{eq: c2def}
				c_2 = \frac{(\sqrt{\alpha}-\sqrt{\beta})^{2}}{K^{d-1} (d-1)!}-\frac{\gamma \log (\alpha / \beta)}{2(d-1)!}-1 >0.
			\end{equation}
			This completes the proof.
		\end{proof}
	\section{Proof of Proposition \ref{prop: localcontraction}}	
		\begin{proof}
			Suppose that $\bm{V}\in \TCal \left(\ACal[\BH^{\otimes(d-1)}]\right)$. From \citet[Lemma 9]{wang2021optimal}, it follows that
			\begin{equation*}
				\bm{VQ}^\top \in \TCal\left(\ACal[\BH^{\otimes(d-1)}]\BQ^\top\right)
			\end{equation*}
			for any $\BQ \in \Pi_K$. In addition, by Lemma \ref{eq: lemma-multi-linear}, Lemma \ref{lemma: lipproj}, and the fact that
			\begin{equation*}
				\ACal[(\BH^*)^{\otimes(d-1)}] \BQ = \ACal[(\BH^*\BQ)^{\otimes(d-1)}], \; \BQ \in \Pi_K,
			\end{equation*}
			we have
			\begin{align*}
				\| \bm{V} - \BH^* \BQ\|_F & = \| \bm{VQ}^\top - \BH^*\|_F
				\leq \frac{2 \left\| \ACal[\BH^{\otimes(d-1)}] \BQ^\top - \ACal[(\BH^*)^{\otimes(d-1)}] \right\|_F }{\gamma \log n}\\
				& = \frac{2 \left\| \ACal[\BH^{\otimes(d-1)}] - \ACal[(\BH^*\BQ)^{\otimes(d-1)}]  \right\|_F }{\gamma \log n} \\
				& \leq \left(\frac{8(d-1)m^{d-2}\varepsilon n}{\sqrt{K}} (p-q) + 2C\sqrt{\log n} \right) \|\BH - \BH^* \BQ\|_F/(\gamma \log n) \\
				& \leq  \left( \frac{8(d-1)\varepsilon (\alpha - \beta)}{\gamma K^{d-3/2}}  +\frac{2C}{\gamma \sqrt{\log n}}\right) \| \BH - \BH^* \BQ\|_F \\
				& \leq  4 \max \left\{\frac{4(d-1)\varepsilon (\alpha - \beta)}{\gamma K^{d-3/2}},  \frac{C}{\gamma \sqrt{\log n}} \right\} \|\BH - \BH^* \BQ\|_F.
			\end{align*}
			This completes the proof.
		\end{proof}
\section{Proof of Theorem \ref{thm: iteration complexity}}		
		\begin{proof}
			The iteration $\BH^1 \in \TCal(\BH^0)$ yields 
			\begin{equation} \label{eq: iteration1step}
				\|\BH^1 - \BH^* \BQ\|_F = \|\BH^1 \BQ^\top - \BH^*\|_F  \le 2 \|\BH^0\BQ^\top - \BH^*\|_F = 2 \|\BH^0- \BH^*\BQ\|_F,
			\end{equation}
   where the inequality comes from Lemma \ref{lemma: lipproj}.
			The following proof is divided into two parts. We first show that for all $t \geq 2$, $\BH^t \in \HCal$ satisfies
			\begin{align} \label{eq: firstpart_goal}
				\| \BH^t - \BH^* \BQ \|_F \leq \frac{1}{2} \| \BH^{t-1} - \BH^*\BQ\|_F \quad \text{ and } \quad  \| \BH^t - \BH^* \BQ\|_F \leq 2 \theta \sqrt{n},
			\end{align}
			and estimate the iteration number $N_1$ such that
			\begin{align}\label{N1:pf-thm-1}
				\|\BH^{N_1} - \BH^* \BQ\|_F \le  2\phi\sqrt{\frac{n}{\log n}}.
			\end{align}
    Suppose that $\BH^0 \in \mathbb{M}_{n,K}$ satisfies \eqref{eq: init-partial}. Combining this initialization condition with \eqref{eq: iteration1step} gives
    \begin{align} \label{eq: h1bound}
        \BH^1 \in \HCal \quad \text{ and } \quad \| \BH^1 - \BH^*\BQ \|_F \leq 2 \theta \sqrt{n}.
    \end{align}
    According to Proposition \ref{prop: localcontraction}, we have
    \begin{align*}
        \| \BH^2 - \BH^* \BQ \|_F & \leq    4 \max \left\{\frac{4(d-1) \theta (\alpha - \beta)}{\gamma K^{d-3/2}},  \frac{C}{\gamma \sqrt{\log n}} \right\}  \| \BH^1 - \BH^* \BQ \|_F\\
        & \leq  4 \max \left\{ \frac{1}{8},  \frac{C}{\gamma \sqrt{\log n}}  \right\}  \| \BH^1 - \BH^* \BQ \|_F
        = \frac{1}{2} \| \BH^1 - \BH^* \BQ \|_F \leq 2\theta \sqrt{n},
    \end{align*}
where the equality comes from $n \geq \exp(64C^2/\gamma^2) $. Therefore, \eqref{eq: firstpart_goal} holds for $t=2$.  By a simple inductive argument, we can show that \eqref{eq: firstpart_goal} holds for $t\ge 3$. Let $N_1 \coloneqq \lceil 2\log\log n \rceil + 1$. It then follows from \eqref{eq: firstpart_goal} that
\begin{align*}
    \|\BH^{N_1} - \BH^* \BQ\|_F &\le \left(\frac{1}{2}\right)^{\lceil 2\log\log n \rceil}\|\BH^1 - \BH^* \BQ\|_F  \le \left(\frac{1}{2}\right)^{ 2\log\log n }2\theta\sqrt{n} \le \left(\frac{1}{2}\right)^{\log\log n + 2\log\left(\frac{\theta}{\phi}\right)}2\theta\sqrt{n} \\
    & \le \left(\frac{1}{2}\right)^{\frac{\log\log n + 2\log\left(\frac{\theta}{\phi}\right)}{2\log 2}}2\theta\sqrt{n} = 2\phi\sqrt{\frac{n}{\log n}},
\end{align*}
where the third inequality is from $n \ge \exp\left(\gamma^2/C^2\right) \ge \exp\left(\theta^2/\phi^2\right)$ and the last inequality comes from $2\log 2 \ge 1$. Thus, \eqref{N1:pf-thm-1} holds for $N_1=\lceil 2\log\log n \rceil + 1$.
   
Next, we show that for all $k \ge 1$, $\BH^{N_1+k} \in \HCal$ satisfies $\|\BH^{N_1+k}-\BH^* \BQ\|_F \le 2\phi\sqrt{n/\log n}$ and
    \begin{align}\label{step2:pf-thm-1}
        \|\BH^{N_1+k} - \BH^*\BQ\|_F \le \frac{4C}{\gamma\sqrt{\log n}} \|\BH^{N_1+k-1}-\BH^*\BQ\|_F,
    \end{align}
    and compute the iteration number $N_2$ such that 
    \begin{align}\label{N2:pf-thm-1}
        \|\BH^{N_2+N_1} - \BH^*\BQ\|_F <  \sqrt{2}. 
    \end{align}
    Since $n\ge \exp\left(\phi^2/\theta^2\right)$, it follows that $2\phi/\sqrt{\log n} \le 4\phi/\sqrt{\log n} \le 4\theta$. According to Proposition \ref{prop: localcontraction}, we obtain
    \begin{align*}
        \|\BH^{N_1+1}-\BH^* \BQ\|_F  \le  4\max\left\{ \frac{8(d-1)\phi (\alpha - \beta)}{\gamma K^{d-3/2}\sqrt{\log n}},\frac{C}{\gamma\sqrt{\log n}} \right\}\|\BH^{N_1}-\BH^* \BQ\|_F
        \le \frac{4C}{\gamma\sqrt{\log n}}\|\BH^{N_1}-\BH^* \BQ\|_F.
    \end{align*}
    Then, \eqref{step2:pf-thm-1} holds for $k=1$. We can show that \eqref{step2:pf-thm-1} holds for $k\ge 2$ by a simple inductive argument. Let $N_2 \coloneqq \left\lceil \frac{2\log n}{\log\log n} \right\rceil$. According to $n \geq \exp(256C^4/\gamma^4)$ and $n \ge \exp\left( 2\phi^2\right)$, we have $\log\log n\ge 4\log(4C/\gamma)$ and $2\phi/\sqrt{\log n} < \sqrt{2}$. This, together with \eqref{step2:pf-thm-1}, yields
    \begin{align*}
        \|\BH^{N_1+N_2} - \BH^* \BQ\|_F  & \le   \left(\frac{4C}{\gamma\sqrt{\log n}} \right)^{\left\lceil \frac{2\log n}{\log\log n} \right\rceil} \|\BH^{N_1}-\BH^* \BQ\|_F  \le 2\phi\sqrt{\frac{n}{\log n}} \left( \frac{4C}{\gamma\sqrt{\log n}} \right)^{\frac{2\log n}{\log\log n}} \\
        & \le 2\phi\sqrt{\frac{n}{\log n}} \left( \frac{4C}{\gamma\sqrt{\log n}} \right)^{\frac{\log n}{\log\log n+2\log(\gamma/(4C))}}  = \frac{2\phi}{\sqrt{\log n}} < \sqrt{2}.
    \end{align*}
    Thus, \eqref{N2:pf-thm-1} holds for $N_2=\left\lceil \frac{2\log n}{\log\log n} \right\rceil$.
    
    Once \eqref{N2:pf-thm-1} holds, we have $\BH^{N_1+N_2}=\BH^* \BQ$. This, together with $\mathcal{T}\left(\ACal[(\BH^*\BQ)^{\otimes(d-1)}]\right) = \{\BH^*\BQ\}$, gives the desired result.
\end{proof}
		
	\end{center}
	\vspace{-0.1in}
	\setcounter{section}{0}
\end{appendix}

\end{document}